\documentclass[10pt]{amsart}
\usepackage{amsfonts}
\usepackage{amssymb}
\usepackage{amsmath}
\usepackage{amsthm}
\usepackage[dvipdfmx]{graphicx}
\usepackage{ascmac}
\usepackage{moreverb}
\usepackage{fancybox}
\usepackage{fancyvrb}
\usepackage{comment}

\theoremstyle{plain}
\newtheorem{theorem}{Theorem}[section]
\newtheorem{lemma}[theorem]{Lemma}
\newtheorem{cor}[theorem]{Corollary}
\newtheorem{prop}[theorem]{Proposition}

\newtheorem{fact}{Fact}
\newtheorem{obs}[theorem]{Observation}

\theoremstyle{definition}
\newtheorem{definition}[theorem]{Definition}
\newtheorem{example}[theorem]{Example}

\theoremstyle{definition}

\newcommand{\upto}{\upharpoonright}
\newcommand{\fr}{\mbox{}^\smallfrown}

\newcommand{\om}{\omega}
\newcommand{\ep}{\varepsilon}
\newcommand{\A}{\mathcal{A}}
\newcommand{\B}{\mathcal{B}}
\newcommand{\F}{\mathcal{F}}
\newcommand{\Q}{\mathcal{Q}}

\newcommand{\si}{\sigma}

\newcommand{\three}{{\mathbf{2}_\bot}}
\newcommand{\pthree}{{\diamondsuit}}
\newcommand{\bfm}{\mathbf{m}}

\newcommand{\De}[1]{{\mathbf{\Delta}^0_{#1}}}
\newcommand{\Si}[1]{{\mathbf{\Sigma}^0_{#1}}}
\newcommand{\PP}[1]{{\mathbf{\Pi}^0_{#1}}}

\newcommand{\sep}[1]{\mathrm{Lev}_{#1}}

\newcommand{\cone}{\triangledown}

\newcommand{\pcsW}{^{\widehat{c}}_{sW}}

\title[Topological reducibilities]{Topological reducibilities for discontinuous functions and their structures}
\author{Takayuki Kihara}
\address[Takayuki Kihara]{Graduate School of Informatics, Nagoya University, Japan}
\email{kihara@i.nagoya-u.ac.jp}
\date{}

\begin{document}
\maketitle

\begin{abstract}
In this article, we give a full description of a topological many-one degree structure of real-valued functions, recently introduced by Day-Downey-Westrick.
We also point out that their characterization of the Bourgain rank of a Baire-one function of compact Polish domain can be extended to noncompact Polish domain.
Finally, we clarify the relationship between the Martin conjecture and Day-Downey-Westrick's topological Turing-like reducibility, also known as parallelized continuous strong Weihrauch reducibility, for single-valued functions:
Under the axiom of determinacy, we show that the continuous Weihrauch degrees of parallelizable single-valued functions are well-ordered; and moreover,
if $f$ is has continuous Weihrauch rank $\alpha$, then $f'$ has continuous Weihrauch rank $\alpha+1$, where $f'(x)$ is defined as the Turing jump of $f(x)$.
\end{abstract}

\section{Introduction}

\subsection{Summary}

The notion of Wadge degrees provides us an ultimate measure to analyze the topological complexity of subsets of a zero-dimensional Polish space (see \cite{And07,AnLo12}).
Under the axiom of determinacy, the induced structure forms a semi-well-order of the height $\Theta$, and thus it enables us to assign an ordinal rank to each subset of such a space.
Our main question is whether one can introduce a similar ultimate measure which induces a semi-well-ordering of {\em real-valued functions} on a Polish space.
A somewhat related question is also proposed by Carroy \cite{Carroy13}.

Recently, Day-Downey-Westrick \cite{DDW} introduced a ``many-one''-like ordering $\leq_\mathbf{m}$ on real-valued functions on Cantor space.
Their ordering $\leq_\mathbf{m}$ measures the topological complexity of sets separating the lower level sets from the upper level sets of a function.
One of our main results in this article is to show that their notion $\leq_\mathbf{m}$ behaves like a Wadge ordering, and in particular, it semi-well-orders real-valued functions.

\begin{definition}[Day-Downey-Westrick \cite{DDW}]\label{def:DDW-reducibi}
For $f,g\colon 2^\om\to\mathbb{R}$, we say that $f$ is $\bfm$-reducible to $g$ (written $f\leq_\bfm g$) if for any rationals $p$ and $\ep>0$, there are rationals $r$ and $\delta>0$ and a continuous function $\theta:2^\om\to 2^\om$ such that, for any $x\in 2^\om$, $g(\theta(x))<r+\delta$ implies $f(x)<p+\ep$, and $g(\theta(x))>r-\delta$ implies $f(x)>p-\ep$.
\end{definition}

One of Day-Downey-Westrick's main discoveries is the connection between their notion of the $\bfm$-degree and the Bourgain rank (also known as the separation rank \cite{KeLo90}) of a Baire-one function.
The latter notion is introduced by Bourgain \cite{Bo80} to prove a refinement of the Odell-Rosenthal theorem in Banach space theory: The $\ell^1$-index of a separable Banach space is related to the degrees of discontinuity (the Bourgain rank) of double-dual elements as Baire-one functions.
Day-Downey-Westrick \cite{DDW} showed the following:
\begin{itemize}
\item The Bourgain rank $1$ consists of exactly two $\bfm$-degrees, those of constants and continuous functions.
\item Every successor Bourgain rank $\geq 2$ consists of exactly four $\bfm$-degrees, where the first two $\bfm$-degrees are incomparable, and the others are comparable.
For instance, the first two $\bfm$-degrees of Bourgain rank $2$ are those of lower semicontinuous and upper semicontinuous functions.
\item Every infinite limit ordinal rank consists of exactly one $\bfm$-degree.
\end{itemize}

Their result completely characterizes the structure of the $\bfm$-degrees of the Baire-one functions, that is, the $\bfm$-degrees of rank below $\om_1$.
In this article, we will give a full description of the structure of the $\bfm$-degrees of all real-valued functions under the axiom of determinacy ${\sf AD}$ (or all Baire-class functions under ${\sf ZFC}$).

\begin{theorem}[{\sf AD}]\label{thm:mthm1}
The $\bfm$-degrees of real-valued functions on $2^\om$ form a semi-well-order of length $\Theta$, where $\Theta$ is the least nonzero ordinal $\alpha$ such that there is no surjection from the reals onto $\alpha$.

For a limit ordinal $\alpha<\Theta$ and finite $n<\omega$, the $\bfm$-rank $\alpha+3n+c_\alpha$ consists of two incomparable degrees, and each of the other ranks consists of a single degree, where $c_\alpha=2$ if $\alpha=0$;  $c_\alpha=1$ if the cofinality of $\alpha$ is $\omega$; and $c_\alpha=0$ if the cofinality of $\alpha$ is uncountable.
\end{theorem}

So far, we have only mentioned functions of Cantor domain.
Now we would like to extend our results to more general domains.
The difficulty arises here by the fact that the structure of Wadge degrees of subsets of a nonzero-dimensional Polish space is ill-behaved (cf.~Ikegami et al.\ \cite{ITS16} and Schlicht \cite{Schl17}).

Fortunately, Pequignot \cite{Peq15} has overcome this difficulty by modifying the definition of Wadge reducibility using the theory of an admissible representation, and then, showed that the modified Wadge degree structure of subsets of a second-countable space is semi-well-ordered.
Day-Downey-Westrick \cite{DDW} adopted a similar idea to consider the notion of $\mathbf{m}$-reducibility for functions of compact metrizable domain.

By integrating their ideas we introduce the notion of $\mathbf{m}$-reducibility for real-valued functions of (quasi-)Polish domain as follows.
Let $\delta$ be a total open admissible representation of a Polish space $\mathcal{X}$ (see Lemma \ref{lem:total-open-admissible}).
Then, we introduce the $\bfm$-degree of a function $f\colon\mathcal{X}\to\mathbb{R}$ as that of $f\circ\delta\colon\om^\om\to\mathcal{X}$.
As in Pequignot \cite{Peq15}, this notion is easily seen to be well-defined (see Section \ref{sec:Bourgain-rank}).
Then we will conclude that the $\bfm$-degrees of real-valued functions on Polish spaces form a semi-well-order of length $\Theta$ (Observation \ref{obs:semi-well-order-m-deg}).

Pequignot's insightful idea also turns out to be very useful for the Wadge-like analysis of the Bourgain rank.
We discuss the Bourgain rank in non-compact spaces (which is also considered by Elekes-Kiss-Vidny\'anszky \cite{EKV16} via the change of topology, in order to generalize the notion of ranks to Baire class $\xi$ functions, and then to study a cardinal invariant associated with systems of difference equations).
Then, based on the notion of sidedness conditions introduced by Day-Downey-Westrick \cite{DDW}, we can classify real-valued functions on a (possibly non-compact) Polish space into (ordered) $5$ types (see Definition \ref{def:DDW-sided-conditions}).
We then generalize the main result in \cite{DDW} to arbitrary Polish domains as follows:
Let $\mathcal{X}$ and $\mathcal{Y}$ be Polish spaces and let $f\colon\mathcal{X}\to\mathbb{R}$ and $g\colon\mathcal{Y}\to\mathbb{R}$ be Baire-one functions.
Then, $f\leq_\bfm g$ if and only if either $\alpha(f)<\alpha(g)$ holds or both $\alpha(f)=\alpha(g)$ and ${\rm type}(f)\leq{\rm type}(g)$ hold.
We also give the precise connection between the Bourgain rank and the Wadge rank.

Finally, we will clarify the relationship between the uniform Martin conjecture and Day-Downey-Westrick's $\mathbf{T}$-degrees of real-valued functions.
They defined $\mathbf{T}$-reducibility for real-valued functions as parallelized continuous strong (p.c.s.)~Weihrauch reducibility, that is, {\em $f$ is $\mathbf{T}$-reducible to $g$} if there are continuous functions $H,K$ such that $f=K\circ\widehat{g}\circ H$, where $\widehat{g}$ is the parallelization of $g$ (see Section \ref{sec:reducibility-notions}).

The Martin conjecture is one of the most prominent open problems in computability theory (see \cite{MarSlaSte}), which generalize Sacks' question on a natural solution to Post's problem.
The notion of $\mathbf{T}$-degree (p.c.s.~Weihrauch degree) is seemingly unrelated to this conjecture; nevertheless we clarify the hidden relationship between them.
We show that the p.c.s.~Weihrauch degrees are exactly the {\em natural} Turing degrees in the context of the uniform Martin conjecture.
More precisely, we will see that the p.c.s.~Weihrauch degrees of real-valued functions is isomorphic to the Turing-degrees-on-a-cone of the uniformly Turing degree invariant operators.
Indeed, the identity map induces an isomorphism between the Turing-ordering-on-a-cone of the uniformly $\leq_T$-preserving operators and the p.c.s.~Weihrauch degrees of real-valued functions.
Therefore, by Steel's theorem \cite{Steel82}, we finally conclude the following.

\begin{theorem}[{\sf AD}]
The p.c.s.~Weihrauch degrees of single-valued functions are well-ordered, whose order type is $\Theta$.
If $f$ is parallelizable (see Section \ref{sec:Martin-conjecture}), and has p.c.s.~Weihrauch rank $\alpha>0$, then $f'$ is also parallelizable, and has p.c.s.~Weihrauch rank $\alpha+1$, where $f'(x)$ is defined as the Turing jump of $f(x)$.
\end{theorem}


\subsection{Conventions and notations}

In Sections \ref{sec:1-3}, \ref{sec:1-4}, \ref{sec:2}, \ref{sec:3} and \ref{sec:Martin-conjecture}, we assume ${\sf ZF}+{\sf DC}+{\sf AD}$ (where ${\sf DC}$ stands for the axiom of dependent choice).
Hence, our results hold in $L(\mathbb{R})$ under ${\sf ZFC}$ plus a large cardinal assumption.
As usual, if we restrict our attention to Borel sets and Baire class functions, every result presented in this article is provable within ${\sf ZFC}$.
If we restrict our attention to projective sets and functions, every result presented in this article is provable within ${\sf ZF}+{\sf DC}+{\sf PD}$ (where ${\sf PD}$ stands for the axiom of projective determinacy).

We assume that $2^\om$ is always embedded into $\mathbb{R}$ as a Cantor set.
For finite strings $\sigma,\tau\in\om^{<\om}$, we write $\sigma\prec\tau$ if $\tau$ extends $\sigma$.
Similarly, for $X\in\om^\om$ we write $\sigma\prec X$ if $X$ extends $\sigma$.
For a string $\sigma$, $[\sigma]$ denotes the set of all $X\in\om^\om$ extending $\sigma$, i.e., $\sigma\prec X$.
Let $X\upto n$ be the initial segment of length $n$.
Let $\sigma\fr\tau$ denote the concatenation of $\sigma$ and $\tau$.

\subsection{The structure of Wadge degrees}\label{sec:1-3}

We here review classical results in the Wadge degree theory \cite{Wadge83,AnLo12}.
For sets $A,B\subseteq\om^\om$, we say that $A$ is Wadge reducible to $B$ (written $A\leq_wB$) if there exists a continuous function $\theta\colon\om^\om\to\om^\om$ such that $A=B\circ\theta$, where we often identify a set with its characteristic function.

Given a pointclass $\Gamma$ (of subsets of $\om^\om$), let $\check{\Gamma}$ denote its dual, that is, $\check{\Gamma}=\{\om^\om\setminus A:A\in\Gamma\}$, and define $\Delta=\Gamma\cap\check{\Gamma}$.
A pointclass $\Gamma$ has the {\em separation property} if
\[(\forall A,B\in\Gamma)\;[A\cap B=\emptyset\;\Longrightarrow\;(\exists C\in\Delta)\;A\subseteq C\mbox{ and }B\cap C=\emptyset].\]
The separation property will play a key role in the proof of our main theorem.

A pointclass $\Gamma$ is {\em self-dual} if $\Gamma=\check{\Gamma}$.
We say that $A\subseteq\om^\om$ is {\em self-dual} if there is a continuous function $\theta\colon \om^\om\to\om^\om$ such that $A(X)\not=A\circ\theta(X)$ for any $X\in\om^\om$.
It is equivalent to saying that $A\leq_w\neg A$.
Note that $A$ is self-dual if and only if the pointclass $\Gamma_A=\{B:B\leq_wA\}$ is self-dual.

By Wadge \cite{Wadge83} and Martin-Monk, non-self-dual pairs are well-ordered, say $(\Gamma_\alpha,\check{\Gamma}_\alpha)_{\alpha<\Theta}$, where $\Theta$ is the height of the Wadge degrees.
We will use the following beautiful fact to show our Main Theorem \ref{thm:mthm1}.

\begin{fact}[Van Wesep \cite{VW78} and Steel \cite{St81}]\label{fact:vanWesep-Steel}
Exactly one of $\Gamma_\alpha$ or $\check{\Gamma}_\alpha$ has the separation property.
\end{fact}

By $\Pi_{\alpha}$, we denote the one which has the separation property, and by $\Sigma_{\alpha}$, we denote the other one (which has the weak reduction property).
Then define $\Delta_\alpha=\Sigma_\alpha\cap\Pi_\alpha$.

A set $A\subseteq\om^\om$ is {\em $\Gamma$-complete} if $A\in\Gamma$ and $B\leq_wA$ for any $B\in\Gamma$.
By definition, a $\Sigma_\alpha$-complete set and a $\Pi_\alpha$-complete set exist for all $\alpha<\Theta$.

\begin{fact}[see \cite{Wesep78}]\label{fact:complete-cofinality}
A $\Delta_\alpha$-complete set exists if and only if the cofinality of $\alpha$ is countable.
\end{fact}

We denote the Borel hierarchy by $(\Si{\alpha},\PP{\alpha},\De{\alpha})_{\alpha<\om_1}$.
More precisely, a set is in $\Si{1}$ if it is open, and a set is in $\Si{\alpha}$ if it is a countable union of sets in $\bigcup_{\beta<\alpha}\PP{\beta}$, where $\PP{\alpha}$ is the dual of $\Si{\alpha}$.
Then, $\De{\alpha}=\Si{\alpha}\cap\PP{\alpha}$.

\begin{example}[Wadge {\cite[Sections V.E and V.F]{Wadge83}}]
The Wadge ranks of sets of finite Borel ranks are calculated as follows.
\begin{itemize}
\item $\Delta_1=$ clopen sets ($=\mathbf{\Delta}^0_1$), $\Sigma_1=$ open sets ($=\mathbf{\Sigma}^0_1$), and $\Pi_1=$ closed sets ($=\mathbf{\Pi}^0_1$).
\item For $\alpha<\om_1$, $\Delta_\alpha$, $\Sigma_\alpha$, and $\Pi_\alpha$ correspond to the $\alpha$-th level of the Hausdorff difference hierarchy.
\item $\Sigma_{\om_1}=F_\sigma$ ($=\mathbf{\Sigma}^0_2$), and $\Pi_{\om_1}=G_\delta$ ($=\mathbf{\Pi}^0_2$).
\item For $\alpha<\om_1$, $\Delta_{\om_1^\alpha}$, $\Sigma_{\om_1^\alpha}$, and $\Pi_{\om_1^\alpha}$ correspond to the $\alpha$-th level of the difference hierarchy over $F_\sigma$.
\item $\Sigma_{\om_1^{\om_1}}=G_{\delta\sigma}$ ($=\mathbf{\Sigma}^0_3$), and $\Pi_{\om_1^{\om_1}}=F_{\sigma\delta}$ ($=\mathbf{\Pi}^0_3$).
\item Generally, $\Sigma_{\om_1\uparrow\uparrow n}=\mathbf{\Sigma}^0_n$, where $\om_1\uparrow\uparrow n$ is the $n$-th level of the superexponential hierarchy of base $\om_1$.
\end{itemize}
\end{example}

Note that the Wadge rank of $\mathbf{\Sigma}^0_{\om}$-complete set is not the first fixed point, but the $\om_1$-th fixed point, of the exponential tower of base $\om_1$.
In general, Wadge {\cite[Sections V.E and V.F]{Wadge83}} has also determined the Wadge ranks of sets of infinite Borel ranks, which are described by using the Veblen hierarchy of base $\om_1$.


Let $\mathcal{X}$ and $\mathcal{Y}$ be topological spaces.
For $A\subseteq\mathcal{X}$ and $B\subseteq\mathcal{Y}$, we write $A/\mathcal{X}\leq_wB/\mathcal{Y}$ if there is a continuous function $\theta\colon \mathcal{X}\to \mathcal{Y}$ such that $X\in A$ if and only if $\theta(X)\in B$.

\begin{lemma}\label{lem:nsd-complete}
For any non-self-dual $A\subseteq\om^\om$, there is $B\subseteq 2^\om$ such that $A/\om^\om\equiv_wB/2^\om$.
\end{lemma}

\begin{proof}
If $A$ is non-self-dual, then $\neg A\not\leq_wA$, and therefore, Player I wins in the Wadge game $G_w(\neg A,A)$.
Put $\hat{\om}=\om\cup\{{\sf pass}\}$.
A winning strategy for Player I gives a continuous function $\theta\colon \hat{\om}^\om\to\om^\om$ such that for any $X\in\hat{\om}^\om$, $\neg A(\theta(X))\not=A(X^{\sf p})$, that is, $A(\theta(X))=A(X^{\sf p})$.
Here, $X^{\sf p}$ is the result of removing all occurrences of {\sf pass}es from $X$.
Define $\eta\colon 2^\om\to\hat{\om}^\om$ by 
\[\eta(0^{n_0}10^{n_1}1\dots)={\sf pass}^{n_0}n_0{\sf pass}^{n_1}n_1\dots\]
Then, define $B(X)=A(\theta\circ\eta(X))$.
We claim that $A/\om^\om\equiv_wB/2^\om$.
Clearly $\theta\circ\eta$ witnesses that $B\leq_wA$.
To see $A\leq_wB$, let $X\in\om^\om$ be a given sequence.
Then, for $\tau(X)=0^{X(0)}10^{X(1)}1\dots$, we have $(\eta\circ\tau(X))^{\sf p}=X$.
Thus,
\[B(\tau(X))=A(\theta\circ\eta\circ\tau(X))=A((\eta\circ\tau(X))^{\sf p})=A(X),\]
where the second equality follows from our choice of $\theta$.
This concludes that $A\leq_wB$.
\end{proof}

\subsection{$\mathcal{Q}$-Wadge degrees}\label{sec:1-4}

A set $A\subseteq\om^\om$ can be identified with its characteristic function $\chi_A\colon\om^\om\to 2$.
Thus, the Wadge degrees of subsets of $\om^\om$ can be viewed as the degrees of $2$-valued functions on $\om^\om$.
The Wadge degrees have been extended in various directions.
For instance, there are various works on the Wadge degrees of partial $2$-valued functions on $\om^\om$ (Wadge \cite{Wadge83}), ordinal-valued functions on $\om^\om$ (Steel, cf.~Duparc \cite{Dup03}), and $k$-partitions of $\om^\om$ (Hertling, cf.~Selivanov \cite{Seli17}).
We can encapsulate all those extensions within the following framework (see also Kihara-Montalb\'an \cite{KMta,KMta2}):

\begin{definition}
Let $(\mathcal{Q};\leq_\mathcal{Q})$ be a quasi-ordered set.
For $\mathcal{Q}$-valued functions $\A,\B\colon\om^\om\to\mathcal{Q}$, we say that {\em $\A$ is $\Q$-Wadge reducible to $\B$} (written $\A\leq_w \B$) if there is a continuous function $\theta\colon\om^\om\to\om^\om$ such that 
\[
(\forall X\in\om^\om)\;\A(X)\leq_\mathcal{Q}\B(\theta(X)).
\]
\end{definition}

As a special case, one can study Wadge's notion of degrees of inseparability of pairs, which will turn out to be a key tool for analyzing the $\mathbf{m}$-degrees.
In his PhD thesis \cite[Section I.E]{Wadge83}, Wadge introduced the notion of reducibility for pairs of subsets of $\om^\om$.
For $A,B,C,D\subseteq\om^\om$, we say that {\em $(A,B)$ is Wadge reducible to $(C,D)$} if there exists a continuous function $\theta\colon \om^\om\to\om^\om$ such that for any $x\in\om^\om$,
\[(x\in A\;\Longrightarrow \theta(x)\in C)\mbox{ and }(x\in B\;\Longrightarrow \theta(x)\in D).\]

Roughly speaking, this reducibility estimates how inseparable a given pair is.
Note that Wadge reducibility for pairs is equivalent to Wadge reducibility for $\{\top,0,1,\bot\}$-valued functions by identifying a pair $(A,B)$ with a function $f_{A,B}$ defined by
\[
f_{A,B}(x)=
\begin{cases}
\top&\mbox{ if }x\in A\cap B,\\
0&\mbox{ if }x\in A\setminus B,\\
1&\mbox{ if }x\in B\setminus A,\\
\bot&\mbox{ if }x\not\in A\cup B,\\
\end{cases}
\]
where $\bot<0,1<\top$, and $0$ and $1$ are incomparable.
It is easy to see that the Wadge degrees of $\{\top,0,1,\bot\}$-valued functions consist exactly of the Wadge degrees of $\{0,1,\bot\}$-valued functions plus a greatest degree, where the greatest degree consists of functions containing $\top$ in their ranges.

Hereafter we use the symbols $\mathbf{2}$ and $\three$ to denote $\{0,1\}$ and $\{0,1,\bot\}$, respectively, where $\mathbf{2}$ is considered as a discrete order, and $\three$ is ordered by $\bot< 0,1$ as mentioned above, which is also known as Plotkin's domain.

Wadge determined the structure of the first few Wadge degrees of inseparability of pairs (equivalently those of $\{0,1,\bot\}$-valued functions).
For $\A\colon \om^\om\to\three$, we define $\neg\A\colon \om^\om\to\three$ by $\A(X)=1-\A(X)$ if $\A(X)\in\{0,1\}$; otherwise $\A(X)=\bot$.
If $\A$ is $\mathbf{2}$-valued, then $\neg\A$ is obviously the complement of $\A$.
Under the axiom of determinacy, Wadge has shown that the semilinear ordering principle holds for $\three$-valued functions.

\begin{fact}[Wadge {\cite[Theorem II.E2]{Wadge83}}]\label{fact:Wadge-SLO}
For any $\A,\B\colon \om^\om\to\three$, either $\B\leq_w\A$ or $\neg\A\leq_w\B$ holds.
\end{fact}

For a function $\A\colon \om^\om\to\Q$ and a finite string $\sigma\in\om^{<\om}$, by $\A\upto[\sigma]$ we denote the restriction of $\A$ up to $[\sigma]$, that is, $(\A\upto[\sigma])(X)=\A(\sigma\fr X)$.
If $\sigma$ is a string of length $1$, $\sigma=\langle n\rangle$ say, then we also write $\A\upto n$ to denote $\A\upto[\langle n\rangle]$.

\begin{definition}
We say that a $\mathcal{Q}$-Wadge degree $\mathbf{a}$ is {\em $\sigma$-join-reducible} if $\mathbf{a}$ is the least upper bound of a countable collection $(\mathbf{b}_i)_{i\in\om}$ of $\mathcal{Q}$-Wadge degrees such that $\mathbf{b}_i<_w\mathbf{a}$.
Otherwise, we say that $\mathbf{a}$ is {\em $\sigma$-join-irreducible}.
\end{definition}

Given a function $\A$ we use the following notation:
\[\F(\A)=\{X:(\forall n)\;\A\upto[X\upto n]\equiv_w\A\}.\]

The following fact gives a better way to characterize $\si$-join-reducibility, which is a straightforward consequence of the well-foundedness of the Wadge degrees (cf.~\cite{AnLo12,Wesep78}).

\begin{fact}\label{fact:sji-deg-pres}
({\sf AD})
A set $\A\subseteq\om^\om$ is $\sigma$-join-irreducible if and only if $\F(\A)$ is nonempty.

A function $\A\subseteq\om^\om$ is $\sigma$-join-reducible if and only if it is Wadge equivalent to a function of the form $\bigoplus_{n\in\om}\A_n$, where each $\A_n$ is $\si$-join-irreducible and $\A_n<_w\A$, and where $\bigoplus_{n\in\om}\A_n$ is defined by $(\bigoplus_{n\in\om}\A_n)(n\fr X)=\A_n(X)$.
\end{fact}

We also need Steel--van Wesep's theorem \cite{Wesep78}.
%

\begin{fact}\label{fact:sjr-equal-sd}
Then a subset of $\om^\om$ is self-dual if and only if it is $\sigma$-join-reducible.
\end{fact}

\section{The structure of $\{0,1,\bot\}$-valued functions}\label{sec:2}

\subsection{The proper $\three$-Wadge degrees}\label{section:proper-three-Wadge}

For $\Gamma\in\{\Sigma,\Pi,\Delta\}$, define $\Gamma_\alpha^\pthree$ to be the class of all $\three$-valued functions which are Wadge reducible to a $\Gamma_\alpha$ set, that is,
\[\Gamma_\alpha^\pthree=\{\A:\om^\om\to\three\mid (\exists S\in\Gamma_\alpha)\;\A\leq_wS\}.\]
Note that $\Delta_\alpha^\pthree=\Sigma_\alpha^\pthree\cap\Pi_\alpha^\pthree$ does {\em not} hold anymore.

Let $\mathcal{D}^\pthree_w$ be the set of all Wadge degrees of $\three$-valued functions.
Then, we define $\mathcal{D}_w\subseteq\mathcal{D}^\pthree_w$ as the set of all Wadge degrees which contain $\mathbf{2}$-valued functions.
A Wadge degree $\mathbf{d}$ is called a {\em proper $\three$-Wadge degree} if $\mathbf{d}\in\mathcal{D}^\pthree_w\setminus\mathcal{D}_w$.
For a Wadge degree $\mathbf{d}\in\mathcal{D}^\pthree_w$, let $\Gamma_\mathbf{d}$ be the collection of all $\A\colon\om^\om\to\three$ such that $\A\leq_w\B$ for some $\B\in\mathbf{d}$.
Note that any $\B\in\mathbf{d}$ is $\Gamma_\mathbf{d}$-complete.

\begin{lemma}\label{lem:insert}
For any proper $\three$-Wadge degree $\mathbf{d}$, there is $\alpha<\Theta$ such that
\[\Delta_\alpha^\pthree\subseteq\Gamma_\mathbf{d}\subseteq\Sigma_\alpha^\pthree\cap\Pi_\alpha^\pthree.\]
\end{lemma}

\begin{proof}
Let $\alpha<\Theta$ be the least ordinal such that $\Gamma_\mathbf{d}\subseteq\Sigma^\pthree_\alpha\cap\Pi^\pthree_\alpha$.
Then, $\Gamma_\mathbf{d}\not\subseteq\Sigma^\pthree_\beta$ or $\Gamma_\mathbf{d}\not\subseteq\Pi^\pthree_\beta$ for any $\beta<\alpha$.
Let $\A$ be a $\Gamma_\mathbf{d}$-complete function, and $B_0,B_1\subseteq\om^\om$ be $\Sigma_\beta$- and $\Pi_\beta$-complete sets, respectively.
Then, $\A\not\leq_wB_i$ for some $i<2$.
By Fact \ref{fact:Wadge-SLO}, we have $B_{1-i}\equiv_w\neg B_i\leq_w\A$.
Since $\mathbf{d}$ is a proper $\three$-Wadge degree, $\A\not\equiv_wB_{1-i}$, and therefore, we also have $\A\not\leq_wB_{1-i}$.
Again, by Fact \ref{fact:Wadge-SLO}, we get $B_i\leq_w\neg B_{1-i}\leq_w\A$.
Therefore, $B_0,B_1\leq_wA$, and hence, we conclude that $\Sigma^\pthree_\beta\cup\Pi^\pthree_\beta\subseteq\Gamma_\mathbf{d}$.

If the cofinality of $\alpha$ is uncountable, then $\Delta_\alpha=\bigcup_{\beta<\alpha}\Sigma_\beta$ since there is no $\Delta_\alpha$-complete set by Fact \ref{fact:complete-cofinality} (see also \cite{Wesep78}).
Therefore, $\Delta_\alpha\subseteq\Gamma_\mathbf{d}$ since $\Sigma_\beta\subseteq\Gamma_\mathbf{d}$ for any $\beta<\alpha$.
Thus, $\Delta_\alpha^\pthree\subseteq\Gamma_\mathbf{d}$.
Assume that the cofinality of $\alpha$ is countable.
Then, by Fact \ref{fact:complete-cofinality}, there is a $\Delta_\alpha$-complete set $C\subseteq\om^\om$.
Since $\Delta_\alpha$ is a selfdual pointclass, $C$ is $\sigma$-join-reducible by Fact \ref{fact:sjr-equal-sd}, and thus by Fact \ref{fact:sji-deg-pres}, one can assume that $C$ is of the form $\bigoplus_nC_n$, where for any $n\in\om$, there is $\beta<\alpha$ such that $C_n\in\Sigma_\beta\cup\Pi_\beta$.
If $\A$ is a $\Gamma_\mathbf{d}$-complete set, then $C_n\leq_w\A$ for any $n$ since $\Sigma_\beta\cup\Pi_\beta\subseteq\Gamma_\mathbf{d}$.
By combining these Wadge reductions, we obtain $\bigoplus_nC_n\leq_w\A$.
Thus, $\Delta_\alpha^\pthree\subseteq\Gamma_\mathbf{d}$.
\end{proof}

%

\subsection{The non-proper $m$-Wadge degrees}\label{sec:structure-non-proper-m-deg}

We say that {\em $\A\colon \om^\om\to\mathcal{Q}$ is $\mathcal{Q}$-$m$-Wadge reducible to $\B\colon \om^\om\to\mathcal{Q}$} ($\A\leq_{mw}\B$) if for any $n\in\om$, there are $m\in\om$ and a continuous function $\theta\colon \om^\om\to\om^\om$ such that for any $X\in\om^\om$,
\[\A(n\fr X)\leq_\mathcal{Q}\B(m\fr\theta(X)),\]

Note that $\theta$ may depend on $n$, and thus the above reduction involves countably many functions $(\theta_n)_{n\in\om}$ and $n\mapsto m$.
Clearly, this is an intermediate notion between Lipschitz reducibility and Wadge reducibility.
The game associated to $\mathcal{Q}$-$m$-Wadge reducibility has been studied by Kihara-Montalb\'an \cite[Section 4.2]{KMta} (which is a simpler version of Steel's degree invariant game \cite{Steel82}).

We say that $\A\colon \om^\om\to\mathcal{Q}$ is {\em $m$-$\sigma$-join-reducible} ($m$-$\sigma$-jr) if $\A\upto n<_w\A$ for any $n\in\om$.
Otherwise, we say that $\A$ is {\em $m$-$\sigma$-join-irreducible} ($m$-$\sigma$-ji).

\begin{lemma}\label{lem:m-deg-sigma-ji}
Assume that $\B\colon \om^\om\to\mathcal{Q}$ is $m$-$\sigma$-ji.
Then, for any $\A\colon \om^\om\to\mathcal{Q}$,
\[\A\leq_{mw}\B\iff \A\leq_w\B.\]
\end{lemma}

\begin{proof}
It is clear that $\A\leq_{mw}\B$ implies $\A\leq_w\B$.
For the reverse implication, since $\B$ is $m$-$\sigma$-ji, there is $m$ such that $\A\leq_w\B\leq_w\B\upto m$.
Let $\theta$ witness $\A\leq_w\B\upto m$.
Then, we have $\A(n\fr X)\leq_\Q\B(m\fr\theta(n\fr X))$.
This clearly implies that $\A\leq_{mw}\B$.
\end{proof}

As $\Sigma_\alpha$ and $\Pi_\alpha$ are non-self-dual, by Fact \ref{fact:sjr-equal-sd}, $\Sigma_\alpha$- and $\Pi_\alpha$-complete sets are $\sigma$-ji.
Thus, by Lemma \ref{lem:m-deg-sigma-ji}, $\Sigma^\pthree_\alpha$- and $\Pi^\pthree_\alpha$-complete functions are also complete w.r.t.\ $m$-Wadge reducibility.

\begin{lemma}\label{lem:m-deg-delta-alpha}
Let $\alpha<\Theta$ be an ordinal of countable cofinality. 
Then, there are exactly two $\three$-$m$-Wadge degrees of the $\Delta^\pthree_\alpha$-complete functions.
\end{lemma}

\begin{proof}
As a $\Delta_\alpha$-complete set is $\sigma$-jr, by Fact \ref{fact:sji-deg-pres}, it is Wadge equivalent to a set of the form $\A=\bigoplus_n\A_n$, which is clearly $m$-$\sigma$-jr, and $\Delta^\pthree_\alpha$-complete.
Then, it is also clear that $0\fr \A=\{0\fr X:X\in \A\}$ is an $m$-$\sigma$-ji $\Delta^\pthree_\alpha$-complete set.
Obviously, $\A<_{mw}0\fr \A$.
Now, $\Delta^\pthree_\alpha$-complete functions split into $m$-$\sigma$-jr ones and $m$-$\sigma$-ji ones.
By Lemma \ref{lem:m-deg-sigma-ji}, if $\A$ and $\B$ are $m$-$\sigma$-ji $\Delta^\pthree_\alpha$-complete functions, then $\A\equiv_{mw}\B$.

Assume that $\A$ and $\B$ are $m$-$\sigma$-jr $\Delta^\pthree_\alpha$-complete functions.
First consider the case that $\alpha$ is a successor ordinal, $\alpha=\beta+1$ say.
By our assumption, we have $\A=\bigoplus_n(\A\upto n)$ where $\A\upto n<_w\A$, which implies that $\A\upto n$ is not $\Delta^\pthree_\alpha$-complete.
If $\A\upto n$ has a proper $\three$-Wadge degree $\mathbf{d}$, then as $\Delta^\pthree_\alpha\not\subseteq\Gamma_\mathbf{d}$, by Lemma \ref{lem:insert}, we have $\A\upto n\in\Gamma_\mathbf{d}\subseteq\Sigma^\pthree_\beta\cap\Pi^\pthree_\beta$.
The join of such functions is also in $\Sigma^\pthree_\beta\cap\Pi^\pthree_\beta$.
This concludes that $\A$ is Wadge equivalent to the join of functions of non-proper $\three$-Wadge degrees; that is, there are $m,n\in\om$ such that $\A\upto m$ and $\A\upto n$ are $\Sigma_\beta^\pthree$- and $\Pi^\pthree_\beta$-complete, respectively.
Again, by Lemma \ref{lem:insert}, for any $n\in\om$, $\B\upto n$ is either $\Sigma^\pthree_\beta$ or $\Pi^\pthree_\beta$ since $\B$ is $m$-$\sigma$-jr.
This shows that $\B\leq_{mw}\A$.
By a symmetric argument, we also have $\A\leq_{mw}\B$, and conclude that $\A\equiv_{mw}\B$.
For a limit ordinal $\alpha$, for any $n\in\om$, $\B\upto n$ is in $\Sigma^\pthree_\beta$ for some $\beta<\alpha$ since $\B$ is m-$\sigma$-jr.
Then there are $m\in\om$ and $\gamma$ such that $\beta\leq\gamma<\alpha$ and every $\Sigma^\pthree_\gamma$ function is $\three$-Wadge reducible to $\A\upto m$ since $\alpha$ is limit and $\A=\bigoplus_n(\A\upto n)$ is $\Delta_\alpha^\pthree$-complete.
Hence, $\B\upto n\leq_w\A\upto m$, which implies that $\B\leq_{mw}\A$.
By a symmetric argument, we also have $\A\leq_{mw}\B$, and conclude that $\A\equiv_{mw}\B$.
\end{proof}

\begin{figure}[t]
	\centering
	\includegraphics{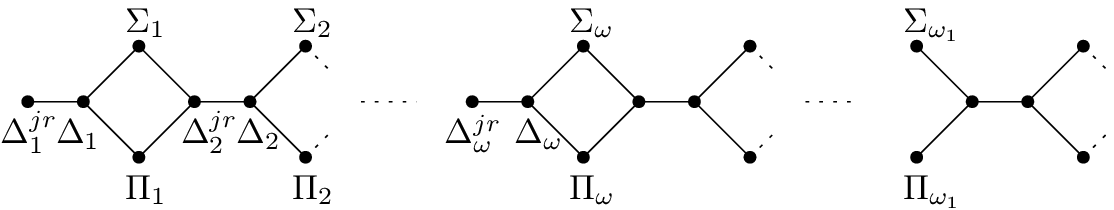}
	\caption{\small The structure of $m$-Wadge degrees of subsets of $\om^\om$ (The structure of {\em non-proper} $m$-Wadge degrees of $\mathbf{2}_\bot$-valued functions of $\om^\om$)}
	\label{fig:1}
\end{figure}

Lemmas \ref{lem:m-deg-sigma-ji} and \ref{lem:m-deg-delta-alpha} give the complete description of the {\em non-proper} $\three$-$m$-Wadge degree structure, hence the $m$-Wadge degree structure of subsets of $\om^\om$.
Each selfdual Wadge degree splits into two degrees (which are linearly ordered), and nonselfdual Wadge degrees remain the same.
See Figure \ref{fig:1}, where $\Delta^{jr}_\alpha$ denotes the class of all sets $m$-Wadge reducible to an $m$-$\sigma$-jr $\Delta_\alpha$ set.

\section{Many-one reducibility for real-valued functions}\label{sec:3}

\subsection{Reducibility for real-valued functions}

Day-Downey-Westrick \cite{DDW} introduced the notion of $\mathbf{m}$-reducibility for real-valued functions.
Let $[\mathbb{Q}]^2$ be the set of all pairs $(p,q)$ of rationals such that $p<q$.
For $f\colon \om^\om\to\mathbb{R}$, we define $\sep{f}\colon [\mathbb{Q}]^2\times\om^\om\to\three$ as follows.
\[
\sep{f}(\langle p,q\rangle\fr X)=
\begin{cases}
0&\mbox{ if }f(X)\leq p,\\
1&\mbox{ if }q\leq f(X),\\
\bot&\mbox{ if }p<f(X)<q.\\
\end{cases}
\]

For $f,g\colon \om^\om\to\mathbb{R}$, we say that {\em $f$ is $\mathbf{m}$-reducible to $g$} (written $f\leq_\mathbf{m}g$) if for any pair of rationals $p<q$, there are a pair of rationals $r<s$ and a continuous function $\theta\colon \om^\om\to\om^\om$ such that for any $X\in\om^\om$,
\[\sep{f}(\langle p,q\rangle\fr X)\leq_{\three}\sep{g}(\langle r,s\rangle\fr\theta(X)).\]

We denote by $\{f\leq p\}$ and $\{f\geq q\}$ the upper and lower level sets $\{X:f(X)\leq p\}$ and $\{X:f(X)\geq q\}$, respectively.
We also define $\{f<p\}$ and $\{f>q\}$ in a similar manner.
In the context of Wadge's pair reducibility \cite[Section I.E]{Wadge83}, $f\leq_\mathbf{m}g$ if and only if for any $p<q$ there are $r<s$ such that $(\{f\leq p\},\{f\geq q\})$ is Wadge reducible to $(\{g\leq r\},\{g\geq s\})$.

\begin{obs}
The above definition of $\leq_\mathbf{m}$ coincides with Definition \ref{def:DDW-reducibi}.
\end{obs}

\begin{proof}
It is clear that $f\leq_\mathbf{m}g$ in the sense of Definition \ref{def:DDW-reducibi} if and only if for any $p,\ep$ there are $r,\delta$ such that $(\{f\leq p-\ep\},\{f\geq p+\ep\})\leq_w(\{g\leq r-\delta\},\{g\geq r+\delta\})$.
\end{proof}


We often identify $[\mathbb{Q}]^2$ and $\om$ via a fixed bijection.
Under this identification, $\sep{f}$ is thought of as a function from $\om\times\om^\om$ to $\three$, and thus, the $\mathbf{m}$-degree structure embeds into the $\three$-$m$-Wadge degree structure, that is,

\begin{obs}
For a function $f\colon \om^\om\to\mathbb{R}$, $f\leq_\mathbf{m}g$ if and only if $\sep{f}\leq_{mw}\sep{g}$.
\qed
\end{obs}

We will show that the map $f\mapsto\sep{f}$ induces an isomorphism between the $\mathbf{m}$-degrees on real-valued functions on $\om^\om$ and the $m$-Wadge degrees on nonempty proper subsets of $\om^\om$.

\begin{theorem}\label{thm:main-theorem1}
The map $f\mapsto {\rm Lev}_f$ induces an isomorphism between the quotients of $(\mathcal{F}(\om^\om,\mathbb{R}),\leq_\mathbf{m})$ and $(\mathcal{P}(\om^\om)\setminus\{\emptyset,\om^\om\},\leq_{mw})$.
\end{theorem}

The proof of Theorem \ref{thm:main-theorem1} will be given in the rest of this section.
As a consequence of Theorem \ref{thm:main-theorem1}, the structure of $\mathbf{m}$-degrees of real-valued functions on $\om^\om$ looks like Figure \ref{fig:1}.
We will also show that the structure of $\mathbf{m}$-degrees of real-valued functions on $2^\om$ looks like Figure \ref{fig:2}.
That is, it is almost isomorphic to the $m$-Wadge degrees of nonempty proper subsets of $\om^\om$ except that if $\alpha$ is a limit ordinal of countable cofinality, the $m$-Wadge degree of an $m$-$\sigma$-ji $\Delta_\alpha$-complete set cannot be realized.

\begin{figure}[t]
	\centering
	\includegraphics{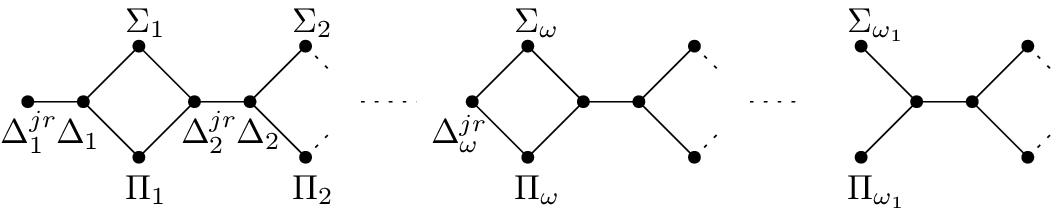}
	\caption{The structure of $m$-degrees of real-valued functions on $2^\om$}
	\label{fig:2}
\end{figure}

\subsection{Non-proper Wadge degrees}

We first characterize non-proper $\three$-Wadge degrees realized as real-valued functions.

\begin{lemma}\label{lem:non-proper-Wadge-for-reals}~
\begin{enumerate}
\item For any nonzero ordinal $\alpha<\Theta$, there are functions $f,g\colon 2^\om\to\{0,1\}$ such that $\sep{f}$ and $\sep{g}$ are $\Sigma_\alpha^\pthree$- and $\Pi_\alpha^\pthree$-complete, respectively.
\item For any nonzero ordinal $\alpha<\Theta$ of countable cofinality, there is a function $f\colon 2^\om\to[0,1]$ such that $\sep{f}$ is $m$-$\sigma$-jr and $\Delta_\alpha^\pthree$-complete.
\item For any successor ordinal $\alpha<\Theta$, there is a function $f\colon 2^\om\to\{0,1\}$ such that $\sep{f}$ is $m$-$\sigma$-ji and $\Delta_\alpha^\pthree$-complete.
\item For any limit ordinal $\alpha<\Theta$ of countable cofinality, there is a function $f\colon \om\times 2^\om\to\{0,1\}$ such that $\sep{f}$ is $m$-$\sigma$-ji and $\Delta_\alpha^\pthree$-complete.
\item For any limit ordinal $\alpha<\Theta$ of countable cofinality, there is no function $f\colon 2^\om\to\mathbb{R}$ such that $\sep{f}$ is $m$-$\sigma$-ji and $\Delta_\alpha^\pthree$-complete.
\item There is no $f\colon \om^\om\to\mathbb{R}$ such that $\sep{f}\equiv_w\emptyset$ or $\sep{f}\equiv_w\om^\om$.
\end{enumerate}
\end{lemma}

\begin{proof}
(1)
Let $A\subseteq 2^\om$ be a $\Sigma_\alpha$-complete set.
Such a set exists by Lemma \ref{lem:nsd-complete}.
Define $\chi_A\colon 2^\om\to\{0,1\}$ by $\chi_A(X)=1$ if $X\in A$; otherwise $\chi_A(X)=0$.
Clearly $A\leq_w\sep{A}:=\sep{{\chi_A}}$ since $A(X)=\sep{A}(\langle 0,1\rangle\fr X)$. 
For $\sep{A}\leq_wA$, since the Wadge rank of $A$ is nonzero, $\emptyset,\om^\om\leq_wA$, and therefore, there is $Y_j\in 2^\om$ such that $A(Y_j)=j$ for each $j\in\{0,1\}$.
Given $p,q$, if $p<0$, define $\theta_{p,q}(X)=Y_0$, and if $q>1$, define $\theta_{p,q}(X)=Y_1$.
If $0\leq p<q\leq 1$, then we have $A(X)=\sep{A}(\langle p,q\rangle\fr X)$, and thus define $\theta_{p,q}(X)=X$.
The function $\theta$ witnesses that $\sep{A}\leq_wA$.
Consequently, $\sep{A}$ is $\Sigma_\alpha^\pthree$-complete.
By a similar argument, one can construct $g$ such that $\sep{g}$ is $\Pi_\alpha^\pthree$-complete.

(2)
First assume that $\alpha$ is a limit ordinal, say $\alpha=\sup_i\beta_i$.
By (1), we have a function $f_i\colon 2^\om\to\{0,1\}$ such that $L_i:=\sep{f_i}$ is $\Sigma_{\beta_i}^\pthree$-complete for each $i<\om$.
Note that $L_i\leq_wL_i\upto\langle 0,1\rangle$.
Let $(a_n)_{n\in\om}$ be strictly decreasing sequence of positive reals.
We define $g(0^\om)=0$ and $g(0^n1X)=a_{2n+1-f_n(X)}$.
Let $A\subseteq\om^\om$ be a $\Delta_\alpha$-complete set of the form $\bigoplus_i A_i$ such that $A_i$ is $\Sigma_{\beta_i}$-complete.
Then, clearly $A$ is $m$-$\sigma$-jr.
We claim that $\sep{g}\equiv_{mw}A$.

For $A\leq_{mw}\sep{g}$, given $n$ we know that $A_n\leq_wL_{n}\upto\langle 0,1\rangle$ via some continuous function $\theta_n$.
Then, 
\[A_n(X)=i\;\Longrightarrow\;f_n(\theta_n(X))=i\;\Longrightarrow\;g(0^n1\theta_n(X))=a_{2n+1-i}.\]

Thus, $\langle n,X\rangle\mapsto 0^n1\theta_n(X)$ witnesses $A\upto\langle 0,1\rangle\leq_{w}\sep{g}\upto\langle a_{2n+1},a_{2n}\rangle$.
Hence, $A\leq_{mw}\sep{g}$.
To see $\sep{g}\leq_{mw}A$, let $p<q$ be rationals.
First consider the case that the open interval $(p,q)$ intersects with $(a_{2n+2},a_{2n+1})$ for some $n\in\om$.
By our definition of $g$, if $X$ extends $0^{n+1}$ then $g(X)\leq a_{2n+2}$, and if $X$ extends $0^m1$ for some $m\leq n$ then $g(X)\geq a_{2n+1}$.
Since the Wadge rank of $A$ is nonzero, $\emptyset,\om^\om\leq_wA$, and therefore, there is $Y_j$ such that $A(Y_j)=j$ for each $j<2$.
Define $\theta_{p,q}(X)=Y_0$ if $X$ extends $0^{n+1}$; otherwise $\theta_{p,q}(X)=Y_1$.
Then,
\begin{multline*}
\sep{g}(\langle p,q\rangle\fr X)=0\;\Longrightarrow\;g(X)\leq p<a_{2n+1}\;\Longrightarrow\;g(X)\leq a_{2n+2}\\
\Longrightarrow\;X\succ 0^{n+1}\;\Longrightarrow\;A\circ\theta_{p,q}(X)=A(Y_0)=0.
\end{multline*}
Similarly, $\sep{g}(\langle p,q\rangle\fr X)=1$ implies that $A\circ\theta_{p,q}(X)=A(Y_1)=1$.
This shows that $\sep{g}\upto\langle p,q\rangle\leq_{w}A\upto\langle 0,1\rangle$.

Now, assume that $(p,q)$ does not intersects with $(a_{2n+2},a_{2n+1})$ for any $n\in\om$.
If $a_0<p$ then $\sep{g}(\langle p,q\rangle\fr X)=0$ for any $X$.
If $q<0$, then $\sep{g}(\langle p,q\rangle\fr X)=1$ for any $X$.
In these cases, it is clear that $\sep{g}\upto\langle p,q\rangle\leq_{w}A\upto\langle 0,1\rangle$.

Otherwise, $a_1\leq p\leq a_0$ or $a_{2n+1}\leq p<q\leq a_{2n}$ for some $n\in\om$.
Assume that $a_{2n+1}\leq p<q\leq a_{2n}$.
If $X$ extends $0^m1$ for some $m<n$, then $g(X)\geq a_{2n-1}>p$, and thus define $\theta_{p,q}(X)=Y_1$.
If $X$ extends $0^{n+1}$, then $g(X)\leq a_{2n+2}<p$, and thus define $\theta_{p,q}(X)=Y_0$.
If $X$ is of the form $0^n1Z$, one can see that $\sep{g}(\langle p,q\rangle\fr X)=L_n(Z)$.
Then, define $\theta_{p,q}(X)=n\fr \tau_n(Z)$, where $\tau_n$ is a continuous function witnesses that $L_n\leq_wA_n$.
This shows that $\sep{g}\upto\langle p,q\rangle\leq_{mw}A\upto\langle 0,1\rangle$.
For $a_1\leq p\leq a_0$, a similar argument applies.
Thus, we conclude that $\sep{g}\equiv_{mw}A$, that is, $\sep{g}$ is $m$-$\sigma$-jr and $\Delta_\alpha^\pthree$-complete.

(4)
Assume that $\alpha=\sup_i\beta_i$.
By (1), we have a function $f_i\colon 2^\om\to\{0,1\}$ such that $L_i:=\sep{f_i}$ is $\Sigma_{\beta_i}$-complete for each $i<\om$.
Note that $L_i\leq_wL_i\upto\langle 0,1\rangle$.
Then, define $g(iX)=f_i(X)$.
It is clear that $\sep{g}$ is $\Delta_\alpha$-complete.
Moreover, one can see that $\sep{g}\leq_w\sep{g}\upto\langle 0,1\rangle$.
Hence, $\sep{g}$ is $m$-$\sigma$-ji.

(3)
A similar argument as in the item (4) also verifies the item (3) except for $\alpha=1$.
Thus, assume that $\alpha=1$, and let $f$ be a nonconstant continuous function.
We claim that $\sep{f}$ is $m$-$\sigma$-ji and $\Delta_1^\pthree$-complete.
Since $f$ is nonconstant, there are $X_0,X_1$ such that $f(X_0)<f(X_1)$.
Choose rationals $r<s$ such that $f(X_0)\leq r<s\leq f(X_1)$.
It suffices to show that $\sep{f}\upto\langle p,q\rangle\leq_w\sep{f}\upto\langle r,s\rangle$ for any $p<q$.
Since $f$ is continuous, $\{f\leq p\}$ and $\{f\geq q\}$ are both closed.
Since $\Pi^0_1$ has the separation property, there is a clopen set $C\subseteq\om^\om$ separating $\{f\leq p\}$ from $\{f\geq q\}$.
Define $\theta_{p,q}(X)=Y_0$ if $X\in C$; otherwise $\theta_{p,q}(X)=Y_1$.
Then, $\theta_{p,q}$ is continuous since $C$ is clopen.
It is easy to see that $\theta_{p,q}$ witnesses that $\sep{f}\upto\langle p,q\rangle\leq_w\sep{f}\upto\langle r,s\rangle$.

(5)
Assume that $\sep{f}$ is $m$-$\sigma$-ji and $\Delta_\alpha^\pthree$-complete.
Since $\sep{f}$ is $m$-$\sigma$-ji, there are $p,q$ such that $\sep{f}\upto\langle p,q\rangle$ is $\Delta_\alpha^\pthree$-complete.
However, since $\sep{f}\upto\langle p,q\rangle$ is a function on $2^\om$, and $\alpha$ is limit, it is impossible by compactness:
If $\A$ is $\Delta^\pthree_\alpha$-complete, then $\A$ is Wadge equivalent to a function of the form $\bigoplus_n\A_n$, where $\A_n\in\Sigma_{\beta_n}^\pthree$ for some $\beta_n<\alpha$.
Let $\theta\colon 2^\om\to\om^\om$ witness $\A\leq_w\bigoplus_n\A_n$.
As $\theta$ is continuous and $2^\om$ is compact, the image of $\theta$ is also compact.
Hence, there is $k\in\om$ such that $\theta$ witnesses $\A\leq_w\bigoplus_{n<k}\A_n$.
This implies that $\A$ is in $\Sigma^\pthree_\beta$, where $\beta=\max_{n<k}\beta_n<\alpha$, a contradiction.

(6)
Fix $X\in\om^\om$.
Then, there are $p,q$ such that $p<f(X)<q$.
Thus, $\sep{f}(\langle p-1,p\rangle\fr X)=1$ and $\sep{f}(\langle q,q+1\rangle\fr X)=0$.
Therefore, $\sep{f}$ is not Wadge reducible to constant functions such as (the characteristic functions of) $\emptyset$ and $\om^\om$.
\end{proof}

\subsection{Proper Wadge degrees}

We finally show that, as a consequence of the van Wesep-Steel Theorem (Fact \ref{fact:vanWesep-Steel}), proper Wadge degrees disappear in the $\mathbf{m}$-degrees of real-valued functions.

\begin{lemma}
For any proper $\three$-Wadge degree $\mathbf{d}$, there is no $f\colon \om^\om\to\mathbb{R}$ such that $\sep{f}$ is $\Gamma_\mathbf{d}$-complete.
\end{lemma}

\begin{proof}
%
Let $\alpha<\Theta$ be an ordinal in Lemma \ref{lem:insert}, that is, $\Delta_\alpha^\pthree\subseteq\Gamma_\mathbf{d}\subseteq\Sigma_\alpha^\pthree\cap\Pi_\alpha^\pthree$.
Assume that $\sep{f}\in\Gamma_\mathbf{d}$.
We claim that for any $p<q$, $\sep{f}\upto\langle p,q\rangle$ is $\Delta_\alpha^\pthree$.
Let $U,V\subseteq\om^\om$ be $\Sigma_\alpha$- and $\Pi_\alpha$-complete sets, respectively.
Choose rationals $p<r<s<q$.
Since $\sep{f}\in\Sigma_\alpha^\pthree\cap\Pi_\alpha^\pthree$, there are continuous functions $\tau_0,\tau_1$ such that
\[\sep{f}(\langle p,r\rangle\fr X)\leq_\three U\circ\tau_0(X)\mbox{, and }\sep{f}(\langle s,q\rangle\fr X)\leq_\three V\circ\tau_1(X).\]
Consider $A=\{X:U\circ\tau_0(X)=0\}$ and $B=\{X:V\circ\tau_1(X)=1\}$, that is, $A=\tau_0^{-1}[\om^\om\setminus U]$ and $B=\tau_1^{-1}[V]$.
Clearly, $A,B\in\Pi_\alpha$.
Moreover, 
\begin{align*}
&f(X)\leq p\;\Longrightarrow\;U\circ\tau_0(X)=0\iff X\in A\;\Longrightarrow\;f(X)<r\\
&f(X)\geq q\;\Longrightarrow\;V\circ\tau_1(X)=1\iff X\in B\;\Longrightarrow\;f(X)>s.
\end{align*}

This implies that $\{f\leq p\}\subseteq A$, $\{f\geq q\}\subseteq B$, and $A\cap B=\emptyset$. 
By the separation property of $\Pi_\alpha$ (Fact \ref{fact:vanWesep-Steel}), there is a $\Delta_\alpha$ set $C\subseteq\om^\om$ such that $C\subseteq A$ and $\om^\om\setminus C\subseteq B$.
Then,
\[\{f\leq p\}\subseteq C\mbox{, and }\{f\geq q\}\subseteq\om^\om\setminus C.\]

This shows that $\sep{f}\upto\langle p,q\rangle\leq_wC$, and therefore, $\sep{f}\upto\langle p,q\rangle$ is $\Delta_\alpha^\pthree$, which verifies our claim.
Consequently, $\sep{f}=\bigoplus_{p,q}\sep{f}\upto\langle p,q\rangle$ is also $\Delta_\alpha^\pthree$, and in particular, $\sep{f}$ cannot be $\Gamma_\mathbf{d}$-complete.
\end{proof}

This concludes the proof of Theorem \ref{thm:main-theorem1}.

\begin{proof}[Proof of Theorem \ref{thm:mthm1}]
For any Wadge degree $\mathbf{d}$ of a subset of $\om^\om$, $\mathbf{d}$ is the Wadge degree of a $\Sigma_\alpha$- or $\Pi_\alpha$-complete set for some $\alpha<\Theta$, or a $\Delta_\alpha$ set for some $\alpha<\Theta$ whose cofinality is countable by Fact \ref{fact:complete-cofinality}.
As seen in the last paragraph in Section \ref{sec:structure-non-proper-m-deg}, every selfdual Wadge degree splits into two $m$-Wadge degrees (which are linearly ordered), and nonselfdual Wadge degrees remains the same.
By Theorem \ref{thm:main-theorem1} and Lemma \ref{lem:non-proper-Wadge-for-reals} (5), we conclude that the structure of $\mathbf{m}$-degrees of real-valued functions on $2^\om$ looks like Figure \ref{fig:2} as desired.
\end{proof}

\section{The Bourgain rank}\label{sec:Bourgain-rank}

In this section, we will generalize Day-Downey-Westrick's result on the Bourgain rank $\alpha$ to an arbitrary Polish domain.
Pequignot \cite{Peq15} extended the notion of Wadge reducibility to subsets of second countable spaces based on the theory of admissible representation.
An {\em admissible representation} of a topological space $\mathcal{X}$ is a partial continuous surjection $\delta:\subseteq\om^\om\twoheadrightarrow\mathcal{X}$ which has the following universal property:
For any partial continuous function $f:\subseteq\om^\om\to\mathcal{X}$, there is a partial continuous function $\tau:\subseteq\om^\om\to\om^\om$ such that $f=\delta\circ\tau$.
An admissible representation is {\em open} if it is an open map, and {\em total} if its domain is $\om^\om$.
The following result seems folklore, but we present the proof for the sake of completeness.

\begin{lemma}\label{lem:total-open-admissible}
Every Polish space has a total open admissible representation.
\end{lemma}

\begin{proof}
If $\mathcal{X}$ is Polish, as in Kechris \cite[Exercise 7.14]{Kec95}, one can construct a Suslin scheme $\mathcal{U}=(U_s)_{s\in\om^{<\om}}$ on $\mathcal{X}$ whose associate map is a total open continuous surjection $\delta_{\rm Sus}^\mathcal{U}\colon \om^\om\twoheadrightarrow\mathcal{X}$, where $\delta_{\rm Sus}^\mathcal{U}(p)$ is defined as a unique element in $\bigcap_nU_{p\upto n}$.
Here, this particular Suslin scheme satisfies $U_{\emptyset}=\mathcal{X}$ and $U_s=\bigcup_nU_{s\fr n}$.
Although it is not hard to check that $\delta_{\rm Sus}^\mathcal{U}$ is admissible, we here give the proof for the sake of completeness.

To see that $\delta_{\rm Sus}^\mathcal{U}$ is admissible, recall that the ``neighborhood filter'' representation is always admissible, that is, for any countable basis $\mathcal{B}=(B_e)_{e\in\om}$ of $\mathcal{X}$ with $B_0=\mathcal{X}$, consider $\mathcal{N}^\mathcal{B}_x=\{e\in\om:x\in B_e\}$, and then define $\delta^\mathcal{B}_{\rm nbhd}(p)=x$ if the range of $p$ is equal to $\mathcal{N}_x^\mathcal{B}$.
It is well-known that $\delta^\mathcal{B}_{\rm nbhd}$ is an admissible representation of $\mathcal{X}$ (cf.~\cite[Theorem 48]{deB13}).

It suffices to find a partial continuous function $\tau:\subseteq\om^\om\to\om^\om$ such that $\delta^\mathcal{B}_{\rm nbhd}=\delta^\mathcal{U}_{\rm Sus}\circ\tau$.
Given $p\in\om^\om$, wait for the first $t\in\om$ such that $B_{p(t)}\subseteq U_{\langle n\rangle}$ for some $n\in\om$.
If there are such $t$ and $n$, define $\tau(p)(0)=n$.
Assume that $s=\tau(p)\upto\ell$ has already been defined.
Then, wait for the first $t\in\om$ such that $B_{p(t)}\subseteq U_{s\fr n}$ for some $n\in\om$.
If there are such $t$ and $n$, define $\tau(p)(\ell)=n$.
Continue this procedure.
Clearly $\tau$ is continuous.
If $\delta^\mathcal{U}_{\rm nbhd}(p)=x\in U_s=\bigcup_nU_{s\fr n}$, one can find $t$ such that $x\in B_{p(t)}\subseteq U_{s\fr n}$ for some $n$ since $\mathcal{N}_x^\mathcal{B}$ is a local basis of $\mathcal{X}$ at $x$.
Therefore, if $p\in{\rm dom}(\delta^\mathcal{U}_{\rm nbhd})$, one can inductively show that $\tau(p)\in\om^\om$, and it is easy to check that $\delta^\mathcal{U}_{\rm nbhd}=\delta^\mathcal{U}_{\rm Sus}(\tau(p))$.
\end{proof}


\begin{obs}\label{obs:admissible-preserve-m-deg}
Let $\mathcal{X}$ be a Polish space.
For any function $f\colon\mathcal{X\to\mathbb{R}}$, the $\mathbf{m}$-degree of $f\circ\delta\colon \om^\om\to\mathbb{R}$ is independent of the choice of a total admissible representation $\delta$.
\end{obs}

\begin{proof}
Let $\delta$ and $\gamma$ be total admissible representations of $\mathcal{X}$.
By admissibility, $\gamma=\delta\circ\theta$ for some continuous function $\theta$.
Thus, $\sep{f\circ\gamma}(\langle p,q\rangle\fr X)=\sep{f}(pq\gamma(X))=\sep{f}(\langle p,q\rangle\fr \delta\circ\theta(X))=\sep{f\circ\delta}(\langle p,q\rangle\fr \theta(X))$.
Hence, $\sep{f\circ\gamma}\leq^\three_{mw}\sep{f\circ\delta}$.
\end{proof}

Hence, given a Polish space $\mathcal{X}$, by Lemma \ref{lem:total-open-admissible}, one can fix a total open continuous admissible representation $\delta$ of $\mathcal{X}$.
Then, the $\mathbf{m}$-degree of $f\colon\mathcal{X}\to\mathbb{R}$ is defined as the $\mathbf{m}$-degree of $f\circ\delta\colon \om^\om\to\mathbb{R}$.

\begin{obs}\label{obs:semi-well-order-m-deg}
The $\bfm$-degrees of real-valued functions on Polish spaces form a semi-well-order of length $\Theta$.
\end{obs}

\subsection{Baire one functions}

\subsubsection{Bourgain rank}

Let $\mathcal{X}$ be a Polish space.
Then, given $f\colon \mathcal{X}\to\mathbb{R}$ and $x\in\mathcal{X}$, we define the following.
\[\overline{f}(x)=\inf_{J\ni x}\sup_{y\in J}f(y),\qquad\underline{f}(x)=\sup_{J\ni x}\inf_{y\in J}f(y),\]
where $J$ ranges over open sets.
It is easy to see that $\underline{f}(x)\leq f(x)\leq\overline{f}(x)$.
Then, given $P\subseteq\mathcal{X}$, the $(f;p,q)$-derivative of $P$ is defined as follows.
\[D_{f,p,q}P=\{x\in P:\underline{(f\upto P)}(x)\leq p<q\leq\overline{(f\upto P)}(x).\}\]

Let $|x|_{f,p,q}$ be the least $\alpha$ such that $x\not\in D^\alpha_{f,p,q}\mathcal{X}$.
Let $\alpha(f,p,q)$ be the least $\alpha$ such that $D^\alpha_{f,p,q}\mathcal{X}=\emptyset$.
Note that $\alpha(f,p,q)=\sup_{x\in\mathcal{X}}|x|_{f,p,q}$.
Then the {\em Bourgain rank of }$f$ is defined by $\alpha(f)=\sup_{p<q}\alpha(f,p,q)$.

By definition, $|x|_{f,p,q}$ is always a successor ordinal.
If the domain is not compact, the rank $\alpha(f,p,q)$ can be a limit ordinal.
In this case, there is no $x$ such that $|x|_{f,p,q}=\alpha(f,p,q)$.
Hence, if the domain is not compact, $\alpha(f)=\alpha(f,p,q)$ can happen even if $\alpha(f)$ is limit.

\subsubsection{Sided-conditions}\label{sec:sided-condition}

Hereafter we use the symbol $P^\nu_{f,p,q}$ to denote $D^\nu_{f,p,q}\mathcal{X}$.

\begin{definition}[see Day-Downey-Westrick \cite{DDW}]\label{def:DDW-sided-conditions}
Let $f\colon \mathcal{X}\to\mathbb{R}$ be a Baire-one function.
\begin{enumerate}
\item $f$ is {\em two-sided} if there are rationals $p<q$ such that $\alpha(f,p,q)=\alpha(f)$ and that for any $\nu<\alpha(f)$, $f(x)\leq p<q\leq f(y)$ for some $x,y\in P^\nu_{f,p,q}$.
\item If $f$ is not two-sided, it is called {\em one-sided}.
\item $f$ is {\em left-sided} if for any rationals $p<q$ with $\alpha(f,p,q)=\alpha(f)$, there is $\nu<\alpha(f)$ such that $P^\nu_{f,p,q}\subseteq\{f<p\}$.
\item $f$ is {\em right-sided} if for any rationals $p<q$ with $\alpha(f,p,q)=\alpha(f)$, there is $\nu<\alpha(f)$ such that $P^\nu_{f,p,q}\subseteq\{f>q\}$.
\item $f$ is {\em irreducible} if there are rationals $p<q$ such that $\alpha(f)=\alpha(f;p,q)$.
Otherwise, $f$ is called {\em reducible}.
\end{enumerate}
\end{definition}

\begin{obs}\label{obs:Bourgain-left-right-reducible}
Let $f\colon \mathcal{X}\to\mathbb{R}$ be a Baire-one function.
\begin{enumerate}
\item If $\alpha(f)$ is a successor ordinal, then $f$ is irreducible.
\item If $\alpha(f)$ is a limit ordinal, and $f$ is irreducible, then $f$ is two-sided.
\item $f$ is both left- and right-sided if and only if $f$ is reducible.
\end{enumerate}
\end{obs}

\begin{proof}
(1)
Since $\alpha(f)=\sup_{p<q}\alpha(f;p,q)$, if $\alpha(f)$, then there must exist $p<q$ such that $\alpha(f)=\alpha(f;p,q)$.

(2)
Let $p<q$ be rationals such that $\alpha(f,p,q)=\alpha(f)$.
For any $\nu<\alpha(f)$ if $\nu<\xi<\alpha(f)$, then there exists $x$ such that $|x|_{f,p,q}=\xi$.
Therefore, for any open neighborhood $J$ of $x$, there are $y,z\in J\cap P^\nu$ such that $f(y)\leq p<q\leq f(z)$.

(3)
By definition, if there is no $p<q$ such that $\alpha(f)=\alpha(f;p,q)$, then $f$ is one-sided, and moreover, left- and right-sided.
For the converse, let $f$ be both left- and right-sided.
Suppose for the sake of contradiction that $f$ is irreducible.
Since $f$ is one-sided and irreducible, by (2), $\alpha(f)$ must be a successor ordinal, $\alpha(f)=\nu+1$ say.
Then $P^\nu_{f,p,q}\not=\emptyset$.
Since $f$ is both left- and right-sided, we also have $P^\nu_{f,p,q}\subseteq\{p<f<q\}$.
Therefore, $x\in P^\nu_{f,p,q}$ implies $p<f(x)<q$, and then let $r$ be such that $f(x)\leq r<q$.
Then, we have $\nu+1=\alpha(f)=\alpha(f;p,q)\leq\alpha(f;r,q)$ and
\[x\in P^\nu_{f,p,q}\subseteq P^\nu_{f,r,q}\not\subseteq\{f>r\}.\]
Therefore, $f$ is not right-sided.
\end{proof}

Define the {\em type of $f$} as follows:
If $f$ is two-sided, define ${\rm type}(f)={\tt t}$.
If $f$ is one-sided, but neither left- nor right-sided, define ${\rm type}(f)={\tt o}$.
If $f$ is left-sided, but not right-sided, define ${\rm type}(f)={\tt l}$.
If $f$ is right-sided, but not left-sided, define ${\rm type}(f)={\tt r}$.
If $f$ is both left- and right-sided (that is, $f$ is reducible by Observation \ref{obs:Bourgain-left-right-reducible} (3)), define ${\rm type}(f)={\tt f}$.
We define the order on these types as follows:
\[{\tt f}<{\tt l},{\tt r}<{\tt o}<{\tt t}.\]

\subsubsection{Open continuous surjection}

\begin{lemma}\label{lem:Bourgain-rank}
Let $\delta\colon\mathcal{Z}\to\mathcal{X}$ be an open continuous surjection, and $f\colon\mathcal{X}\to\mathbb{R}$ be a Baire-one function.
Then, $|x|_{f\circ\delta,p,q}=|\delta(x)|_{f,p,q}$.
In particular, $\alpha(f)=\alpha(f\circ\delta)$ and ${\rm type}(f)={\rm type}(f\circ\delta)$.
\end{lemma}

\begin{proof}
We claim that $\delta^{-1}[D_{f,p,q}P]=D_{f\circ\delta,p,q}\delta^{-1}[P]$ for any $P\subseteq\mathcal{X}$.
If $z\in\delta^{-1}[D_{f,p,q}P]$, then there is an open neighborhood $J$ of $\delta(z)$ in $P$ such that $f(x)\leq p$ and $f(y)\geq q$ for some $x,y\in J\cap P$.
Then, $\delta^{-1}[J]$ is an open neighborhood of $z$ in $\delta^{-1}[P]$, and since $\delta$ is surjective, there are $u,v\in\delta^{-1}[J]\cap\delta^{-1}[P]$ such that $\delta(u)=x$ and $\delta(v)=y$.
Since $f\circ\delta(u)\leq p$ and $f\circ\delta(v)\geq q$, we have $z\in D_{f\circ\delta,p,q}\delta^{-1}[P]$.
Conversely, if $z\in D_{f\circ\delta,p,q}\delta^{-1}[P]$, then there is an open neighborhood $J$ of $z$ in $\delta^{-1}[P]$ such that $f\circ\delta(x)\leq p$ and $f\circ\delta(y)\geq q$ for some $x,y\in J\cap\delta^{-1}[P]$.
Since $\delta$ is an open map, $\delta[J]$ is an open neighborhood of $\delta(z)$, and we also have $\delta(x),\delta(y)\in\delta[J]\cap P$.
Consequently, $\delta(z)\in D_{f,p,q}P$, and thus $z\in\delta^{-1}[D_{f,p,q}P]$.
This verifies the claim.

We will inductively show that $\delta^{-1}[D_{f,p,q}^\xi\mathcal{X}]=D_{f\circ\delta,p,q}^\xi\mathcal{Z}$ for any $\xi$.
It is obvious for a limit step.
For a successor ordinal $\xi+1$,
\[D_{f\circ\delta,p,q}^{\xi+1}\mathcal{Z}=D_{f\circ\delta,p,q}\delta^{-1}[D_{f,p,q}^\xi\mathcal{X}]=\delta^{-1}[D_{f,p,q}D_{f,p,q}^\xi\mathcal{X}]=\delta^{-1}[D_{f,p,q}^{\xi+1}\mathcal{X}].\]

The first equality follows from the induction hypothesis, and the second equality follows from the previous claim with $P=D_{f,p,q}^\xi\mathcal{X}$.
Consequently, $x\in D^\xi_{f\circ\delta,p,q}\mathcal{Z}$ iff $\delta(x)\in D^\xi_{f,p,q}\mathcal{X}$, and hence $|x|_{f\circ\delta,p,q}=|\delta(x)|_{f,p,q}$ as desired.
\end{proof}

\subsubsection{The Day-Downey-Westrick Theorem}

\begin{theorem}\label{thm:Day-Downey-Westrick}
Let $\mathcal{X}$ and $\mathcal{Y}$ be Polish spaces and let $f\colon \mathcal{X}\to\mathbb{R}$ and $g\colon \mathcal{Y}\to\mathbb{R}$ be Baire-one functions.
Then, $f\leq_\mathbf{m}g$ if and only if either $\alpha(f)<\alpha(g)$ holds or both $\alpha(f)=\alpha(g)$ and ${\rm type}(f)\leq{\rm type}(g)$ hold.
\end{theorem}

\begin{proof}
By Observation \ref{obs:admissible-preserve-m-deg} and Lemma \ref{lem:Bourgain-rank}, one can assume that $\mathcal{X}=\mathcal{Y}=\om^\om$ by considering $f\circ\delta_\mathcal{X}$ and $g\circ\delta_\mathcal{Y}$ instead of $f$ and $y$, where $\delta_\mathcal{X}$ and $\delta_\mathcal{Y}$ are total open admissible representations of $\mathcal{X}$ and $\mathcal{Y}$ ensured by Lemma \ref{lem:total-open-admissible}, respectively.

For the left-to-right direction, it is not hard to check that a straightforward modification of the argument in Day-Downey-Westrick \cite{DDW} gives us the desired condition.

For the right-to-left direction, assume that either $\alpha(f)<\alpha(g)$ holds or both $\alpha(f)=\alpha(g)$ and ${\rm type}(f)\leq{\rm type}(g)$ hold.
Given $p<q$ we need to find $r<s$ satisfying the following property:
For any $x$, there is $y$ such that
\begin{align}\label{equ:DDW-initial-condition}
|x|_{f,p,q}\leq |y|_{g,r,s},\mbox{ and }{\rm Lev}_f(x;p,q)\leq_{\mathbf{2}_\bot}{\rm Lev}_g(y;r,s).
\end{align}

If $\alpha(f;p,q)<\alpha(g)$, then there are $r<s$ such that $\alpha(f;p,q)<\alpha(g;r,s)$.
Choose such $r<s$.
In particular, if either $\alpha(f)<\alpha(g)$ or ${\rm type}(f)={\tt f}$, then we get (\ref{equ:DDW-initial-condition}).
Now assume $\alpha(f)=\alpha(g)$ and ${\rm type}(f)\not={\tt f}$.
In this case, there are $p<q$ such that $\alpha(f;p,q)=\alpha(f)=\alpha(g)$.

%

The rest of the proof is done by the argument as in Day-Downey-Westrick \cite{DDW}.
\end{proof}

Note that Lemma \ref{lem:Bourgain-rank} only requires $\delta$ to be an open continuous surjection (that is, $\delta$ is not necessarily admissible).
Hence, Theorem \ref{thm:Day-Downey-Westrick} implies that the $\mathbf{m}$-degree of a Baire-one function $f\circ\delta$ is independent of the choice of an open continuous surjection $\delta$.

%
%
%

One of the most important conclusions of Theorems \ref{thm:mthm1} and \ref{thm:Day-Downey-Westrick} is that one can characterize the Bourgain rank in terms of the descriptive complexity as follows:

\begin{cor}\label{cor:Wadge-Bourgain}
Let $\delta\colon \om^\om\to\mathcal{X}$ be an open continuous surjection, and $f\colon \mathcal{X}\to\mathbb{R}$ be a Baire-one function.
Then,
\begin{enumerate}
\item $\sep{f\circ\delta}$ is either $\Sigma_\xi^\pthree$- or $\Pi_\xi^\pthree$-complete if and only if $\alpha(f)=\xi+1$ and either $f$ is left- or right-sided.
\item $\sep{f\circ\delta}$ is $\Delta_\xi^\pthree$-complete and $m$-$\sigma$-ji if and only if $\alpha(f)=\xi$ and $f$ is two-sided.
\item $\sep{f\circ\delta}$ is $\Delta_\xi^\pthree$-complete and $m$-$\sigma$-jr if and only if $\alpha(f)=\xi$ and either ${\rm type}(f)={\sf o}$ (if $\xi$ is successor) or ${\rm type}(f)={\sf f}$ (if $\xi$ is limit).
\end{enumerate}
\end{cor}

\begin{proof}
By Lemma \ref{lem:Bourgain-rank}, we can assume that $\mathcal{X}=\om^\om$ and $\delta={\rm id}$.
By induction.
If $\alpha(f)$ is successor, then ${\rm type}(f)\in\{{\tt l},{\tt r},{\tt o},{\tt t}\}$ by Observation \ref{obs:Bourgain-left-right-reducible} (1).
If $\alpha(f)$ is limit, then ${\rm type}(f)\in\{{\tt f},{\tt t}\}$ by Observation \ref{obs:Bourgain-left-right-reducible} (2) and (3).
Moreover, it is easy to see that every such type is realized by some function whenever $\alpha(f)>1$.
One can see that if $\alpha(f)=1$ then ${\rm type}(f)\in\{{\tt o},{\tt t}\}$, that is, either $f$ is constant or $f$ is nonconstant and continuous.

Let $\xi$ be a limit ordinal.
If $\xi=0$, put ${\tt c}={\tt o}$; otherwise put ${\tt c}={\tt f}$.
By Theorem \ref{thm:Day-Downey-Westrick}, one can easily see that for any $n\in\om$,
\begin{align*}
\mbox{${\rm rank}_\mathbf{m}(f)=\xi+3n$} & \iff \mbox{$\alpha(f)=1+\xi+n$ and ${\rm type}(f)={\tt c}$}\\
\mbox{${\rm rank}_\mathbf{m}(f)=\xi+3n+1$} & \iff \mbox{$\alpha(f)=1+\xi+n$ and ${\rm type}(f)={\tt t}$}\\
\mbox{${\rm rank}_\mathbf{m}(f)=\xi+3n+2$} & \iff \mbox{$\alpha(f)=1+\xi+n+1$ and ${\rm type}(f)\in\{{\tt l},{\tt r}\}$,}
\end{align*}
where ${\rm rank}_\mathbf{m}(f)$ denotes the $\mathbf{m}$-rank of $f$.
The proof of Theorems \ref{thm:mthm1} and Theorem \ref{thm:main-theorem1} shows that
\begin{align*}
\mbox{${\rm rank}_\mathbf{m}(f)=\xi+3n$} & \iff \mbox{$\sep{f}$ is $\Delta_{1+\xi+n}^\pthree$-complete and $m$-jr}\\
\mbox{${\rm rank}_\mathbf{m}(f)=\xi+3n+1$} & \iff \mbox{$\sep{f}$ is $\Delta_{1+\xi+n}^\pthree$-complete and $m$-ji}\\
\mbox{${\rm rank}_\mathbf{m}(f)=\xi+3n+2$} & \iff \mbox{$\sep{f}$ is $\Sigma_{1+\xi+n}^\pthree$- or $\Pi_{1+\xi+n}^\pthree$-complete.}
\end{align*}

This concludes the proof.
\end{proof}

\subsubsection{Separation rank}

We characterize the Bourgain rank in the context of the Hausdorff difference hierarchy by modifying the $\alpha_1$-rank introduced by Elekes-Kiss-Vidny\'anszky \cite{EKV16}.
For $\mathbf{\Gamma}\in\{\mathbf{\Sigma},\mathbf{\Pi},\mathbf{\Delta}\}$, let $\xi$-$\mathbf{\Gamma}^0_1$ be the corresponding pointclasses in the $\xi^{th}$ level of the Hausdorff difference hierarchy, that is, $\xi$-$\mathbf{\Gamma}^0_1$ is equal to $\Gamma_\xi$.
For a Baire-one function $f\colon\om^\om\to\mathbb{R}$ and rationals $p<q$, let $\alpha^{\rm sep}_1(f;p,q)$ be the least ordinal $\alpha$ such that an $\alpha$-$\mathbf{\Delta}^0_1$ set separates $\{f\leq p\}$ from $\{f\geq q\}$.
Then, we define
\[\alpha^{\rm sep}_1(f)=\sup_{p<q}\alpha^{\rm sep}_1(f;p,q).\]

\begin{obs}\label{obs:sep-and-Wadge}
Let $f\colon \om^\om\to\mathbb{R}$ be a Baire-one function.
Then, $\alpha^{\rm sep}_1(f)\leq\xi$ iff $\sep{f}\leq^\three_{mw}E_\xi$, where $E_\xi$ is an $m$-$\sigma$-ji $\Delta_\xi$-complete subset of $\om^\om$.
\end{obs}

\begin{proof}
It is clear that a $\xi$-$\mathbf{\Delta}^0_1$ set (i.e., a $\Delta_\xi$ set) separates $\{f\leq p\}$ and $\{f\geq q\}$ for any $p,q$ if and only if $\sep{f}$ is $\three$-Wadge reducible to a $\Delta_\xi$-complete set, and it is equivalent to saying that $\sep{f}$ is $\three$-$m$-Wadge reducible to an $m$-$\sigma$-ji $\Delta_\xi$-complete set.
\end{proof}

Generally, for a Baire-one function $f\colon\mathcal{X}\to\mathbb{R}$, the {\em pre-separation rank} $\alpha^{\rm sep}(f)$ is defined by
\[\alpha^{\rm sep}(f)=\min_\delta\alpha_1^{\rm sep}(f\circ\delta),\]
where $\delta$ ranges over open continuous surjections from $\om^\om$ onto $\mathcal{X}$.

\begin{prop}\label{prop:Bourgain-equal-sep}
$\alpha(f)=\alpha^{\rm sep}(f)$.
\end{prop}

\begin{proof}
Let $\delta$ be an open continuous surjection witnessing that $\alpha^{\rm sep}(f)=\alpha_1^{\rm sep}(f\circ\delta)$.
By Corollary \ref{cor:Wadge-Bourgain}, if $\xi$ is successer, $\xi=\eta+1$ say, $\alpha(f)=\xi$ iff $S_{f\circ\delta}$ is $\Gamma$-complete for some $\Gamma\in\{\Sigma_\eta,\Pi_\eta,\Delta_{\eta+1}^{jr},\Delta_{\eta+1}\}$.
If $\xi$ is limit, $\alpha(f)=\xi$ iff $S_{f\circ\delta}$ is either $\Delta_{\xi}^{jr}$- or $\Delta_{\xi}$-complete.
In any case, $\alpha(f)\leq\xi$ iff $S_{f\circ\delta}\leq^\three_{mw}E_\xi$.
By Observation \ref{obs:sep-and-Wadge}, the latter is equivalent to saying that $\alpha^{\rm sep}_1(f\circ\delta)\leq\xi$.
By our choice of $\delta$, we have $\alpha^{\rm sep}(f)=\alpha^{\rm sep}_1(f\circ\delta)$.
Consequently, $\alpha(f)\leq\xi$ iff $\alpha^{\rm sep}(f)\leq\xi$.
\end{proof}

By Lemma \ref{lem:Bourgain-rank}, the above proposition shows that $\alpha(f)=\alpha^{\rm sep}_1(f\circ\delta)$ holds for any open continuous surjection $\delta$.
In particular, this implies that the definition of $\alpha^{\rm sep}(f)$ is independent of the choice of $\delta$.

\section{$\mathbf{T}$-degrees and the Martin ordering}\label{sec:Martin-conjecture}

In this section, we will see a strong connection between the structure of Day-Downey-Westrick's $\mathbf{T}$-degrees of real-valued functions \cite{DDW} and the {\em Martin ordering on the uniform Turing degree invariant functions}.
Then, by combining with Becker's result \cite{Becker88}, we will see that the $\mathbf{T}$-degrees of real-valued functions form a well-order of type $\Theta$.

\subsection{Reducibility notions}\label{sec:reducibility-notions}

There are a number of works on Wadge-like classifications of functions on $\om^\om$.
For instance, Carroy \cite{Carroy13} adopted {\em continuous strong Weihrauch reducibility} as a tool to provide a reasonable classification of functions on $\om^\om$, where for $f,g\colon \om^\om\to\om^\om$, we say that $f$ is continuously strongly Weihrauch reducible to $g$ (written as $f\leq^c_{sW}g$) if there are continuous functions $\Phi,\Psi\colon \om^\om\to\om^\om$ such that
\[f=\Phi\circ g\circ\Psi.\]
Subsequently, Day-Downey-Westrick \cite{DDW} adopted {\em parallelized continuous strong (p.c.s.) Weihrauch reducibility} as a formalization of topological ``Turing reducibility'' for real-valued functions.
Given a function $h\colon \mathcal{X}\to\mathcal{Y}$, define the {\em parallelization of $h$} as the following function $\widehat{h}\colon \mathcal{X}^\om\to\mathcal{Y}^\om$:
\[\widehat{h}(\langle x_n:n\in\om\rangle)=\langle h(x_n):n\in\om\rangle.\]

We use $\leq\pcsW$ to denote the p.c.s.~Weihrauch reducibility, that is, $f\leq\pcsW g$ iff $f\leq_{sW}^c\widehat{g}$.
%
%
It is equivalent to say that there are continuous functions $\Phi:\subseteq\mathbb{R}^\om\to\mathbb{R}$ and $\Psi\colon \om\times\om^\om\to\om^\om$ such that
\[(\forall X\in\om^\om)\;f(X)=\Phi\left(\langle g(\Psi(i,X)):i\in\om\rangle\right).\]

We connect the reducibility notion $\leq\pcsW$ with the {\em uniform Martin conjecture} \cite{Steel82,SlaSte88}.
To explain this, we need to introduce several notions from computability theory.
For $X,Y\in 2^\om$, we say that $Y$ is Turing reducible to $X$ (written $Y\leq_TX$) if there is a partial computable function $\Phi:\subseteq 2^\om\to 2^\om$ such that $\Phi(X)=Y$.
We write $X\equiv_TY$ if $X\leq_TY$ and $Y\leq_TX$.
We fix an effective enumeration of all partial computable functions.
If $\Phi$ in the definition of Turing reducibility is given as the $e$-th partial computable function, then we say that $Y\leq_TX$ via $e$.

In 1960s, Martin conjectured that {\em natural} Turing degrees are well-ordered, and the successor rank is given by the Turing jump.
Usually, ``natural'' means that relativizability and Turing invariance, but we here also require the uniformity on Turing invariance.

\begin{definition}
A function $f\colon 2^\om\to 2^\om$ is {\em uniformly Turing degree invariant} (UI) if there is a function $u\colon \om^2\to\om^2$ such that
\[X\equiv_TY\mbox{ via }(d,e)\;\Longrightarrow\;f(X)\equiv_Tf(Y)\mbox{ via }u(d,e).\]

A function $f\colon 2^\om\to 2^\om$ is {\em uniformly Turing order preserving} (UOP) if there is a function $u\colon \om\to\om$ such that
\[X\leq_TY\mbox{ via }e\;\Longrightarrow\;f(X)\leq_Tf(Y)\mbox{ via }u(e).\]

For functions $f,g\colon  2^\om\to 2^\om$, we say that {\em $f$ is Martin reducible to $g$} (or {\em $f$ is Turing reducible to $g$ a.e.}; written $f\leq_T^\cone g$) if
\[(\exists C\in 2^\om)(\forall X\geq_TC)\;f(X)\leq_Tg(X).\]
\end{definition}

It is clear that every UOP function is UI.
The converse also holds up to the Turing equivalence a.e.

\begin{fact}[Becker \cite{Becker88}]
Every UI function is Turing equivalent to a UOP function a.e.
\end{fact}

However, this reducibility notion $\leq_T^\cone$ is badly behaved for constant functions.
In this article, we use the following variant of $\leq_T^\cone$:

\begin{definition}
For functions $f,g\colon  2^\om\to 2^\om$, we write $f\leq_\mathbf{T}^\cone g$ if
\[(\exists C\in 2^\om)(\forall X\geq_TC)\;f(X)\leq_Tg(X)\oplus C.\]
\end{definition}

A function $f\colon 2^\om\to 2^\om$ is {\em increasing a.e.}\ if there is $C\in 2^\om$ such that $f(X)\geq_TX$ for all $X\geq_TC$.
A function $f\colon 2^\om\to 2^\om$ is {\em constant a.e.}\ if there is $C\in 2^\om$ such that $f(X)\equiv_TC$ for all $X\geq_TC$.

\begin{fact}[Slaman-Steel \cite{SlaSte88}]\label{fact:Slaman-Steel}
For a UI function $f\colon 2^\om\to 2^\om$, either $f$ is constant a.e.\ or $f$ is increasing a.e.
\end{fact}

\begin{obs}\label{obs:bold-T-equal-light-T}
Let $f,g\colon  2^\om\to 2^\om$ be UI functions.
If $g$ is not constant a.e., then
\[f\leq_T^\cone g\iff f\leq_\mathbf{T}^\cone g.\]
\end{obs}

\begin{proof}
Assume that $f\leq_\mathbf{T}^\cone g$ via $C$.
By Fact \ref{fact:Slaman-Steel}, $g$ is increasing a.e.
Therefore, there is $D\geq_TC$ such that $g(X)\geq_TX$ for all $X\geq_TD$.
For such $X$, $f(X)\leq_Tg(X)\oplus C\leq_T g(X)$.
Hence, $f\leq_T^\cone g$.
\end{proof}

The $\leq^\cone_\mathbf{T}$-degrees of UOP functions forms a well-order of height $\Theta$, and the successor rank is given by the Turing jump (cf.~Steel \cite{Steel82} and Becker \cite{Becker88}).

By ${\rm UOP}$ we denote the collection of UOP functions, and by $\mathcal{F}$ we denote the collection of real-valued functions on $\om^\om$.
In this section, we will show the following.

\begin{theorem}\label{thm:pcsW-equal-Martin}
The identity map induces an isomorphism between quotients of $({\rm UOP},\leq_\mathbf{T}^\cone)$ and $(\mathcal{F},\leq\pcsW)$.
\end{theorem}

%
%

As a corollary, by Observation \ref{obs:bold-T-equal-light-T}, the identity map induces an isomorphism between the Martin ordering on the UOP operators which is not constant a.e.\ and the parallel continuous strong Weihrauch degrees of real-valued non-constant functions.
Theorem \ref{thm:pcsW-equal-Martin} also concludes the following.

\begin{theorem}\label{thm:pcsW-equal-Martin2}
The p.c.s.~Weihrauch degrees of real-valued functions on $\om^\om$ form a well-order of type $\Theta$.
Moreover, if $g\colon 2^\om\to 2^\om$ has nonzero p.c.s.~Weihrauch rank $\alpha$, then $\widehat{g}'$ has p.c.s.~Weihrauch rank $\alpha+1$, where $h'(x)$ is defined as the Turing jump of $h(x)$.
\end{theorem}

In particular, for any parallelizable function $g$ (that is, $g\equiv_{sW}^c\widehat{g}$), if $g$ has p.c.s.~Weihrauch rank $\alpha$, then $g'$ has p.c.s.~Weihrauch rank $\alpha+1$.

\subsection{Injectivity}

We first show that the identity map gives an embedding of $({\rm UOP},\leq_T^\cone)$ into $(\mathcal{F},\leq\pcsW)$.
We will use the following notion.
A {\em uniformly pointed perfect tree} (u.p.p.\ tree) is a perfect tree $T\subseteq 2^{<\om}$ such that $T\leq_TT[X]$ via some index $e$ independent of $X\in 2^\om$, where we often think of a perfect tree $T$ as a continuous embedding $T[\cdot]\colon 2^\om\to 2^\om$, that is, $T[X]$ is the $X$-th infinite path through $T$.

\begin{fact}[Martin; see \cite{MarSlaSte}]\label{fact:Martin}
For any countable partition $(P_i)_{i\in\om}$ of $2^\om$, there is $i\in\om$ such that $P_i$ includes the set of all infinite paths through a u.p.p.\ tree.
\end{fact}

\begin{lemma}\label{lem:UOP-parallelizable}
Assume that $f$ and $g$ are UOP functions.
Then,
\[f\leq_\mathbf{T}^\cone g\iff f\leq_{sW}^cg \iff f\leq_{sW}^c\widehat{g}.\]
\end{lemma}

\begin{proof}
Assume that $f\leq_\mathbf{T}^\cone g$ via $C$.
By Fact \ref{fact:Martin}, there are a u.p.p.\ tree $T$ and an index $e$ such that for any $X$, $f(T[X])\leq_Tg(T[X])\oplus C$ via $\Phi_e$.
Note that $\Phi_e^C\colon Z\mapsto\Phi_e(Z\oplus C)$ is continuous.
Assume that $f$ is UOP via $u$.
For an index $d$ witnessing $X\leq_TT[X]$, we have $f(X)\leq_Tf(T[X])$ via $\Phi_{u(d)}$.
Then, we have
\[f(X)=\Phi_{u(d)}(f(T[X]))=\Phi_{u(d)}(\Phi_e(g(T[X])\oplus C))=\Phi_{u(d)}(\Phi_e^C(g(T[X]))).\]
This concludes that $f=\Phi_{u(d)}\circ\Phi_e^C\circ g\circ T$, and thus, $f\leq_{sW}^cg$ as desired.

Conversely, assume that $f\leq_{sW}^c\widehat{g}$.
Then, there are continuous functions $\Phi,\Psi$ such that $f(X)=\Phi(\langle g(\Psi(i,X))\rangle_i)$ for all $X$.
Let $C$ be an oracle such that $\Phi$ and $\Psi$ are $C$-computable.
If $X\geq_TC$, then $\Psi(i,X)$ is $X$-computable uniformly in $i$, that is, there is a computable function $p$ such that $\Psi(i,X)\leq_TX$ via $\Phi_{p(i)}$.
Let $u$ witness that $g$ is UOP.
Then, if $X\geq_TC$, then we have $g(\Psi(i,X))\leq_Tg(X)$ via $\Phi_{u\circ p(i)}$.
Therefore, $\bigoplus_ig(\Psi(i,X))\leq_Tg(X)\oplus u$.
Then, for any $X\geq_TC$,
\[f(X)=\Phi\left(\langle g(\Psi(i,X))\rangle_i\right)\leq_Tg(X)\oplus u\oplus C.\]
Consequently, we get that $f\leq_\mathbf{T}^\cone g$.
\end{proof}

\subsection{Surjectivity}

To prove Theorem \ref{thm:pcsW-equal-Martin}, it remains to show that every function $f\colon \om^\om\to\mathbb{R}$ is $\equiv\pcsW$-equivalent to a UOP function.
Clearly, every constant function is UOP, and any two constant functions are $\equiv_\mathbf{T}^\cone$-equivalent, and $\equiv\pcsW$-equivalent.
We hereafter assume that $f\colon \om^\om\to\mathbb{R}$ is not constant.

\subsubsection{Continuous functions}

By Theorem \ref{thm:main-theorem1}, the Wadge degrees of $\three$-valued functions of the form $\sep{f}$ can be identified with the Wadge degrees of subsets of $\om^\om$.
For $\mathbf{m}$-degrees, recall that each selfdual degree splits into two degrees, but this spliting happens only for $\mathbf{m}$-degrees.
Actually, one can see that parallel continuous strong Weihrauch reducibility for non-constant functions is coarser than Wadge reducibility as follows.

\begin{lemma}\label{lem:wadge-implies-pcsW}
Assume that $g$ is not constant.
If $\sep{f}\leq_w\sep{g}$, then $f\leq_{sW}^c\widehat{g}$.
\end{lemma}

To show Lemma \ref{lem:wadge-implies-pcsW}, we need the following sublemma.
We write $f\leq_W^cg$ if $f\leq_{sW}^c({\rm id},g)$, where given functions $f,h$, define $(f,h)\colon x\mapsto(f(x),h(x))$.

\begin{lemma}\label{lem:remove-strong}
Let $f,g\colon \om^\om\to\mathbb{R}$ be functions.
Assume that $g$ is not constant.
Then, $f\leq_{sW}^c\widehat{g}$ if and only if $f\leq_{W}^c\widehat{g}$.
\end{lemma}

\begin{proof}
It suffices to show that $({\rm id},\widehat{g})\leq_{sW}^c\widehat{g}$.
Since $g$ is not constant, there are $Y_0,Y_1\in\om^\om$ such that $g(Y_0)\not=g(Y_1)$.
Let $U_0,U_1$ be disjoint rational open intervals such that $g(Y_i)\subseteq U_i$ for each $i<2$.
Given $X\in 2^\om$, define $\rho(i,X)=Y_{X(i)}$.
We also define $\tau(\bigoplus_i Z_i)(n)=i$ if $Z_n\in U_i$.
Then,
\[X=\tau\left(\bigoplus_ig(Y_{X(i)})\right)=\tau\left(\bigoplus_ig\circ\rho(i,X)\right).\]
This shows that ${\rm id}\leq_{sW}^c\widehat{g}$.
Consequently, $({\rm id},\widehat{g})\leq_{sW}^c(\widehat{g},\widehat{g})\equiv_{sW}\widehat{g}$.
\end{proof}

Note that the outer reduction $\tau$ in the proof of Lemma \ref{lem:remove-strong} is clearly computable.
Later we will use this observation to show Theorem \ref{thm:pcsW-equal-Martin2}.

\begin{proof}[Proof of Lemma \ref{lem:wadge-implies-pcsW}]
Assume that $\sep{f}\leq_w\sep{g}$.
Then, there are continuous functions $r,s\colon \mathbb{Q}^2\times\om^\om\to\mathbb{Q}$ and $\psi\colon \mathbb{Q}^2\times\om^\om\to\om^\om$ such that for any $p<q$ and $X\in\om^\om$,
\begin{align*}
f(X)\leq p\;&\Longrightarrow\;g(\psi(p,q,X))\leq r(p,q,X),\\
f(X)\geq q\;&\Longrightarrow\;g(\psi(p,q,X))\geq s(p,q,X).
\end{align*}
Then, one can show that there is a continuous function $\Phi$ such that $f(X)=\Phi(X,\bigoplus_{p,q}g(\psi(p,q,X)))$ for all $X$.
Consequently, $f\leq_{W}^c\widehat{g}$, and thus  $f\leq_{sW}^c\widehat{g}$ by Lemma \ref{lem:remove-strong}.
\end{proof}

By Lemma \ref{lem:wadge-implies-pcsW}, the $\equiv\pcsW$-degrees of continuous functions consist only of two degrees, that is, if $f$ and $g$ are continuous, but not constant, then $f\equiv\pcsW g$.
In particular, $f\equiv\pcsW {\rm id}$, where note that the identity map ${\rm id}$ is clearly a UOP function.
Hence, it remains to consider the case that $f$ is discontinuous.

\subsubsection{Nonselfdual functions}

Assume that $\sep{f}$ is nonselfdual.
As in Becker \cite{Becker88}, we first assign a UOP function to each nonselfdual Wadge degree.
Following Becker \cite[Definiton 2.2]{Becker88}, we say that a pointclass $\Gamma$ is {\em reasonable} if $\Gamma$ is $\om$-parametrized, contains all computable sets, and has the substitution property.
Given a $\Gamma$-indexing $U$, we define $J_\Gamma^U\colon 2^\om\to 2^\om$ as follows:
\[J^U_\Gamma(X)=\{\langle m,n\rangle:U(m,n,X)\}.\]

\begin{fact}[Becker {\cite[Lemmas 2.5 and 2.6]{Becker88}}]
For any reasonable pointclass $\Gamma$ and its indexing $U$, $J_\Gamma^U$ is a UOP function which is increasing a.e.
Moreover, the $\equiv_T^\cone$-degree of $J_\Gamma^U$ is independent of the choice of $U$.
\end{fact}

As in Section \ref{section:proper-three-Wadge}, given a pointclass $\Gamma$, we define $\Gamma^\pthree=\{\A:\om^\om\to\three\mid(\exists S\in\Gamma)\;\A\leq_wS\}$.

\begin{lemma}
For any reasonable pointclass $\Gamma$, if $\sep{f}$ is $\Gamma^\pthree$-complete, then $f\equiv\pcsW J_\Gamma$.
\end{lemma}

\begin{proof}
Let $U$ be a $\Gamma$-indexing, which is, in particular, $\Gamma$-complete.
Since $\sep{f}\in\Gamma^\pthree$, there is a continuous function $\theta$ such that $\sep{f}(\langle p,q\rangle\fr X)\leq_\three U\circ\theta(p,q,X)$.
Then $\theta(p,q,X)$ is of the form $(\tau_{pq}(X),\Psi_{pq}(X))$, where $\tau_{pq}(X)\in\om^2$ and $\Psi_{pq}(X)\in\om^\om$.
Thus, $\sep{f}(\langle p,q\rangle\fr X)\leq_\three J^U_\Gamma(\Psi_{pq}(X))(\tau_{pq}(X))$.
Then, one can construct a continuous function $\Phi$ such that $f(X)=\Phi(\bigoplus_{p,q}J^U_\Gamma\circ\Psi_{pq}(X))$.
Consequently, $f\leq_{sW}^c\widehat{J_\Gamma}$.

Conversely, since $U\in\Gamma$ and $\sep{f}$ is $\Gamma^\pthree$-complete, there is a continuous function $\eta$ such that $U=\sep{f}\circ\eta$.
By a similar argument as above, one can show that $J_\Gamma\leq_{sW}^c\widehat{f}$.
\end{proof}

Becker \cite[Lemma 3.4]{Becker88} showed that for every nonzero ordinal $\alpha<\Theta$, there are reasonable pointclasses $\Sigma$ and $\Pi$ such that $\mathbf{\Sigma}=\Sigma_\alpha$ and $\mathbf{\Pi}=\Pi_\alpha$.
Consequently, if $\sep{f}$ is nonselfdual, then there is a UOP jump operator $g$ such that $f\equiv_{sW}^{\widehat{c}}g$.

\subsubsection{Selfdual functions}

It remains to consider the case that $\sep{f}$ is selfdual, and $f$ is discontinuous.
By discontinuity of $f$, we have $\sep{f}\not\in\Delta_1^\pthree$.
Generally, the following lemma states that we do not need to deal with a selfdual Wadge degree of successor rank.

\begin{lemma}
Assume $\alpha>0$.
If $\sep{f}\in\Delta^\pthree_{\alpha+1}$ and if $\sep{g}$ is $\Sigma^\pthree_\alpha$-complete, then $f\leq_{sW}^c\widehat{g}$.
\end{lemma}

\begin{proof}
Given $g\colon \om^\om\to\mathbb{R}$, define $-g$ by $(-g)(X)=-g(X)$.
Note that $\sep{-g}$ is $\Pi^\pthree_\alpha$-complete whenever $\sep{g}$ is $\Sigma^\pthree_\alpha$-complete.
Define $h$ by $h(0X)=g(X)$ and $h(1X)=-g(X)$.
Then, $\sep{h}$ is $\Delta^\pthree_{\alpha+1}$-complete, and therefore, if $\sep{f}\in\Delta^\pthree_{\alpha+1}$, then $\sep{f}\leq_wS_h$.
Thus, by Lemma \ref{lem:wadge-implies-pcsW}, we have $f\leq_{sW}^c\widehat{h}$.
Therefore, it suffices to show that $h\leq_{sW}^c\widehat{g}$.
Since $\alpha>0$, $g$ cannot be constant, that is, there are $Z_0,Z_1$ such that $g(Z_0)\not=g(Z_1)$.
Let $U_0,U_1$ be disjoint open sets such that $g(Z_i)\subseteq U_i$ for each $i<2$.
Define $\Psi(0iX)=Z_i$, and $\Psi(1iX)=X$.
Then define $\Phi(z\oplus y)=y$ if $z\in U_0$, and $\Phi(z\oplus y)=-y$ if we find that $z\in U_1$.
Then, we get 
\[h(iX)=\Phi(g(\Psi(0iX))\oplus g(\Psi(1iX))).\]
Consequently, $h\leq_{sW}^c\widehat{g}$ as desired.
\end{proof}

If $\alpha$ is a limit ordinal of countable cofinality, there is a sequence $\beta_n<\alpha$ such that $\alpha=\sup_n\beta_n$.
Let $J_{\beta_n}$ be a UOP function corresponding to the reasonable pointclass $\Sigma_{\beta_n}$.
Then, as in Becker \cite{Becker88}, define $J_\alpha$ as follows:
\[J_\alpha(X)=\bigoplus_{n\in\om} J_{\beta_n}(X).\]

Becker \cite{Becker88} showed that $J_\alpha$ is a UOP function on a u.p.p.\ tree, that is, there is a u.p.p\ tree $T$ such that $J_\alpha^\ast:=J_\alpha\circ T$ is UOP.

\begin{lemma}
For any limit ordinal $\alpha$ of countable cofinality, if $\sep{f}$ is $\Delta_\alpha^\pthree$-complete, then $f\equiv_{sW}^{\widehat{c}}J_\alpha^\ast$.
\end{lemma}

\begin{proof}
Straightforward.
\end{proof}

This concludes the proof of Theorem \ref{thm:pcsW-equal-Martin}.
We finally show Theorem \ref{thm:pcsW-equal-Martin2} saying that the jump of the parallelization always gives the successor rank.

\begin{proof}[Proof of Theorem \ref{thm:pcsW-equal-Martin2}]
As mentioned before, the $\leq^\cone_\mathbf{T}$-degrees of UOP functions form a well-order of height $\Theta$, and therefore, by Theorem \ref{thm:pcsW-equal-Martin}, so are the p.c.s.~Weihrauch degrees.

We claim that for non-constant functions $f,g\colon 2^\om\to 2^\om$, if $f\leq\pcsW g$ then $f'\leq_{sW}^c \widehat{g}'$.
First note that the Turing jump $X\mapsto X'$ is UOP, so let $u$ be a witness of UOP-ness of the Turing jump, that is, if $X\leq_TY$ via $e$ then $X'\leq_TY'$ via $u(e)$.
If $f\leq\pcsW g$ then there are continuous functions $h$ and $k$ such that $f=k\circ\widehat{g}\circ h$.
Put $B_X=\widehat{g}(h(X))$ for any $X\in 2^\om$.
As $k$ is continuous, there is an oracle $C$ such that $k$ is $C$-computable.
Hence, $f(X)\leq_T B_X\oplus C$ via some index $e$ independent of $X$.
As in the proof of Lemma \ref{lem:remove-strong}, since $g$ is not constant, one can find $\langle Z_n^C:n\in\om\rangle$ such that $\widehat{g}(\langle Z_n^C:n\in\om\rangle)$ computes $C$.
Assume that $h(X)=\langle h_n(X):n\in\om\rangle$.
Then define $r_{2n}(X)=h_n(X)$, $r_{2n+1}(X)=Z_n^C$, and $r(X)=\langle r_n(X):n\in\om\rangle$.
Since $B_X\oplus C\leq_T\widehat{g}(r(X))$ via some index $d$ (independent of $X$), $(B_X\oplus C)'\leq_T\widehat{g}(r(X))'$ via $u(d)$.
Moreover, since $f(X)\leq_TB_X\oplus C$ via $e$, we have $f(X)'\leq_T(B_X\oplus C)'$ via $u(e)$.
Hence, there is an index $c$ independent of $X$ such that $f(X)'\leq_T\widehat{g}(r(X))'$ via $c$.
In other words, the pair $(r,\Phi_c)$ witnesses $f'\leq_{sW}^c\widehat{g}'$.

Let $f$ be a non-constant function of p.c.s.~Weihrauch rank $\alpha$.
By Theorem \ref{thm:pcsW-equal-Martin}, there is a UOP function $g$ such that $f\equiv\pcsW g$, and the $\leq^\cone_\mathbf{T}$-rank of $g$ is also $\alpha$.
It is easy to see that $\widehat{g}$ and $\widehat{g}'$ are also UOP, and clearly $\widehat{g}\equiv\pcsW g$.
Hence, by Theorem \ref{thm:pcsW-equal-Martin}, the $\leq^\cone_\mathbf{T}$-rank of $\widehat{g}$ is still $\alpha$.
Then, by Steel's theorem \cite{Steel82}, the $\leq^\cone_\mathbf{T}$-rank of $\widehat{g}'$ is $\alpha+1$, and again by Theorem \ref{thm:pcsW-equal-Martin}, so is the p.c.s.~Weihrauch rank.
By the above claim, we have $\widehat{f}'\equiv\pcsW\widehat{g}'$.
Consequently, the p.c.s.~Weihrauch rank of $\widehat{f}'$ is $\alpha+1$
\end{proof}

The claim in the above proof also shows that if $f$ is non-constant and parallelizable, so is $f'$:
By parallelizability of $f$, we have $\widehat{f}\leq_{sW}^c f$, and by the above claim, we also have $\widehat{f}'\leq_{sW}^c f'$.
For $A_n=f(X_n)$, as the Turing jump is UOP, $\bigoplus_{n\in\om}A_n'$ is computable in $(\bigoplus_{n\in\om}A_n)'$ in a uniform manner; hence $\widehat{(f')}\leq_{sW}^c \widehat{f}'$.
Therefore, $f'$ is parallelizable.

As a consequence, if $f$ is non-constant and parallelizable, then one can obtain an $\om_1$-sequence of p.c.s.~Weihrauch successor ranks of $f$ only by iterating the jump $g\mapsto g'$.

\subsection*{Acknowledgements}
The author was partially supported by JSPS KAKENHI, Grant-in-Aid for Research Activity Start-up, Grant Number 17H06738.
The author also thanks JSPS Core-to-Core Program (A. Advanced Research Networks) for supporting the research.
The author would like to thank Adam Day, Rod Downey, Antonio Montalb\'an, and Linda Brown Westrick for valuable discussions.

\bibliographystyle{plain}
\bibliography{topological_reducibilities}

\end{document}